\documentclass[11pt]{article}

\usepackage{amsmath,amssymb,amsthm}
\usepackage{enumitem}
\usepackage{geometry}
\newcommand{\F}{\mathbb{F}}
\newcommand{\C}{\mathbb{C}}
\newcommand{\rad}{\mathrm{rad}}
\newcommand{\Fun}{\mathrm{Fun}}
\newcommand{\End}{\mathrm{End}}

\numberwithin{equation}{subsection}

\newtheorem{theorem}{Theorem}[subsection]
\newtheorem{lemma}[theorem]{Lemma}
\newtheorem{proposition}[theorem]{Proposition}
\newtheorem{corollary}[theorem]{Corollary}
\newtheorem{example}[theorem]{Example}

\theoremstyle{definition}
\newtheorem{definition}[theorem]{Definition}
\newtheorem{remark}[theorem]{Remark}

\newcommand{\Comp}{\mathrm{Cpl}}

\usepackage{tikz}
\usetikzlibrary{arrows.meta, positioning, calc, shapes.geometric}

\title{Boundary Calculus, Rigidity Islands, and Deformation Theory in Algebraic Phase Structures}
\author{
Joe Gildea\\
Department of Computing Science and Mathematics,\\
School of Informatics and Creative Arts,\\
Dundalk Institute of Technology\\
\texttt{gildeajoe@gmail.com}}
\date{}

\begin{document}
\maketitle

\begin{abstract}
We develop a general boundary calculus for algebraic phases and use it to
formulate an intrinsic structural framework for deformation and obstruction
phenomena. Structural boundaries are shown to be finitely detectable and
canonically stratified by failure type and depth. For each boundary we
construct a canonical boundary exact sequence and identify a maximal rigid
subphase, called a rigidity island, that persists beyond global
boundary failure.

Rigidity islands are organised by intrinsic invariants and serve as canonical
base points for deformation theory. Deformation behaviour within the standing
admissibility framework is governed by boundary quotients, while rigidity
islands remain stable under admissible deformation. Boundary quotients act as
obstruction objects whose associated strata organise higher-depth deformation
behaviour.

As a consequence, deformation behaviour is naturally stratified by boundary
depth and failure type, while formal smoothness is associated with the
vanishing of boundary data. The resulting moduli behaviour is organised by
rigidity islands together with their associated obstruction patterns, without
requiring intrinsic analytic or continuous deformation parameters.
\end{abstract}

\medskip



\section{Introduction}

Algebraic Phase Theory (APT) provides a structural framework for extracting,
organizing, and classifying algebraic data arising from interaction laws and
defect propagation. Rather than beginning with a fixed algebraic object, APT
starts from admissible interaction data and studies how canonical algebraic
structure, defect, filtration, and finite termination emerge intrinsically.

Earlier developments in the theory
\cite{GildeaAPT1,GildeaAPT2,GildeaAPT3,GildeaAPT4} revealed a hierarchical
structural picture. Interaction generates defect data; defect induces a
canonical filtration; termination leads to structural boundaries; and, in the
strongly admissible regime, many apparent distinctions between models collapse
under intrinsic phase equivalence and Morita-type equivalence. These phenomena
appear across radical, representation-theoretic, and quantum coding examples
developed in the earlier papers of the series.

What has remained less understood is the internal structure of boundaries and
their role in deformation and obstruction phenomena. Earlier work identifies
where boundaries occur and establishes their invariance properties, but does
not systematically analyse how boundaries are detected intrinsically at finite
depth, how they decompose into canonical strata, or how they organise
rigidity, obstruction, deformation, and moduli behaviour.

The purpose of this paper is to develop such a framework. We introduce a
boundary calculus in which structural boundaries are detected and stratified,
and we study how these strata organise structural variation in terminating
algebraic phases. A central theme is that boundaries do not simply represent a
loss of control. Instead, boundary formation reorganises surviving algebraic
structure into rigidity islands that are intrinsic, functorial within the
standing admissibility framework, and independent of presentation. These rigid
components provide canonical anchors for reconstruction and deformation theory.

This perspective differs substantially from classical deformation frameworks,
including Gerstenhaber deformation of algebras, Kodaira--Spencer theory,
tensor categorical deformation, and Hopf algebraic approaches
(cf.~\cite{AndersonFuller1992Rings,Faith1973Algebra,Lam2001FirstCourse,
EtingofGelaki2015Tensor,ConnesKreimer1998Hopf}). Classical theories typically
rely on linearisation, cohomology, or analytic or formal parameters, whereas
deformation behaviour in APT is organised through boundary strata and finite
boundary depth. Obstruction phenomena arise through the interaction between
boundary layers and defect propagation, and the resulting deformation
behaviour is naturally stratified by boundary type and depth.

The structural behaviour identified here places APT in dialogue with classical
rigidity and representation-theoretic frameworks
(cf.~\cite{Jacobson1956Structure,Mackey1958Induced,Serre,Howe1979Heisenberg,
CurtisReiner,Weil1964SurCertains,Honold2001QuasiFrobenius,Wood1999Duality}),
while remaining independent of additive, analytic, or cohomological
foundations. In particular, the boundary calculus developed here connects the
invariance and rigidity phenomena established earlier in the theory with the
reconstruction and duality questions developed in subsequent work
\cite{GildeaAPT6}.

Taken together, these results motivate the view that structural boundaries act
as intrinsic horizons for canonical propagation: beyond them, global
functorial control may fail, while local rigid substructures continue to
persist in a controlled and intrinsically detectable way. The boundary
calculus developed in this paper provides a framework for organising the
associated obstruction, deformation, and moduli phenomena in Algebraic Phase
Theory.

\section{Standing Framework and Structural Assumptions}

This paper is purely structural. We do not repeat the analytic or pre-algebraic
extraction procedures of Algebraic Phase Theory developed in Paper~I
\cite{GildeaAPT1}. Instead, we fix the minimal internal framework assumed
throughout and work entirely within the structural consequences of those assumptions.

An \emph{algebraic phase} is a structured object $(\mathcal P,\circ)$ encoding
a specified family of interaction operations. All defect, filtration, rigidity, and complexity data are not additional
structure; they are induced functorially from $(\mathcal P,\circ)$ within the
standing admissibility assumptions. Throughout the paper, all
phases are assumed to satisfy the axioms of Algebraic Phase Theory, and we
recall only their structural content.

Two admissibility regimes play a central role. An algebraic phase is
\emph{strongly admissible} if all defect propagation and higher interaction
constraints determined by $(\mathcal P,\circ)$ are forced functorially at every
filtration depth. It is \emph{weakly admissible} if defect detection and
canonical filtration are functorially determined at finite depth, but higher
propagation may fail to be canonical beyond a controlled threshold. Strong admissibility yields rigidity throughout the controlled defect regime, whereas weak admissibility permits
boundary behavior beyond defect generation.

Phase interaction determines a functorial notion of defect, which measures
deviation from rigidity. Defect induces a canonical filtration measuring
structural complexity, and this filtration always terminates after finitely
many steps. Under the standing admissibility assumptions, the theory admits a canonical
descending sequence
\[
\mathcal P = \mathcal P^{(0)} \supseteq \mathcal P^{(1)} \supseteq
\mathcal P^{(2)} \supseteq \cdots,
\]
a distinguished rigid core $\mathcal P_{\mathrm{rig}}$, and intrinsic boundary
quotients associated with the breakdown of controlled extension beyond rigidity.
All structural consequences used in this paper depend only on these intrinsic
features and not on the manner in which the phase was originally extracted or
realized.

The \emph{defect rank} $d(\mathcal P)$ is the maximal defect degree required to
generate all non-rigid behavior of $\mathcal P$ under its interaction
operations. It controls the point at which defect propagation ceases to
determine structure functorially.

Unlike Papers~I–III \cite{GildeaAPT1,GildeaAPT2,GildeaAPT3}, no admissible
phase data, analytic realizations, or operator models are constructed here.
Unlike Paper~IV \cite{GildeaAPT4}, the emphasis is not on categorical
organization. The focus of the present work is the intrinsic calculus of structural
boundaries, the existence and classification of rigidity islands, and the
deformation and obstruction theory associated with boundary data.
The constructions and invariants considered in this paper are formulated to be
compatible with intrinsic phase equivalence and Morita-type relations.

\section{Structural Boundaries Revisited}

This section develops the notion of structural boundary and analyzes
its relationship to defect generation, termination, and intrinsic
invariance. We begin by introducing boundary depth and comparing it
to defect rank.

\subsection{Boundary Depth and Defect Generation}

Structural boundaries are the central organizing principle of the present
paper. They mark the precise points at which the axiomatic mechanisms of Algebraic
Phase Theory cease to propagate structure functorially.
Rather than representing pathological failure, boundaries are intrinsic,
detectable, and rigid features of phase interaction.

Conceptually, a boundary is not an external obstruction imposed on a phase,
but an internal limit associated with defect-induced complexity.
Beyond this limit, canonical extension is no longer possible, and new
phenomena such as rigidity islands, obstruction objects, and controlled
deformation may emerge. We now formalize this notion.

\begin{definition}
Let $\mathcal P$ be a terminating algebraic phase with canonical filtration
\[
\mathcal P=\mathcal P^{(0)} \supseteq \mathcal P^{(1)} \supseteq \cdots
\supseteq \mathcal P^{(L)}=0.
\]
A \emph{structural boundary} occurs at filtration depth $k$ if:
\begin{itemize}
  \item $\mathcal P^{(k)}$ is the minimal stratum at which at least one axiom of
  Algebraic Phase Theory fails to propagate functorially, and
  \item all strata $\mathcal P^{(i)}$ for $i<k$ satisfy the axioms fully.
\end{itemize}
The integer $k$ is called the \emph{boundary depth}.
\end{definition}

\begin{proposition}
Let $\mathcal P$ be a terminating algebraic phase with defect rank $d$ and
boundary depth $k$.
Then necessarily
\[
k \ge d.
\]
Moreover, if $k>d$, one may write
\[
k = d + k_{\mathrm{ext}},
\qquad k_{\mathrm{ext}} \ge 0,
\]
where $k_{\mathrm{ext}}$ is the \emph{weak extension depth}: the number of
filtration steps beyond defect generation for which canonical propagation of
higher interaction constraints remains functorial before the first failure.
\end{proposition}

\begin{proof}
Let
\[
\mathcal P=\mathcal P^{(0)} \supseteq \mathcal P^{(1)} \supseteq \cdots
\supseteq \mathcal P^{(L)}=0
\]
be the canonical terminating filtration.

By definition of \emph{defect rank} $d=d(\mathcal P)$, all non-rigid behavior of
$\mathcal P$ is generated by defect data of degree at most $d$, and $d$ is
minimal with this property.

By definition of \emph{boundary depth} $k=k(\mathcal P)$, all axioms of Algebraic
Phase Theory propagate functorially on $\mathcal P^{(i)}$ for $i<k$, and at least
one axiom fails to propagate functorially at depth $k$.

Since defect generation up to degree $d$ is required to account for all non-rigid
behavior, boundary formation cannot occur strictly before defect generation.
Hence $k\ge d$.

If $k>d$, then defect generation completes by depth $d$ and, by minimality of $k$,
all axioms continue to propagate functorially on each stratum $\mathcal P^{(d+j)}$
for $0\le j<k-d$ before failing for the first time at depth $k$. In this case set
\[
k_{\mathrm{ext}}:=k-d,
\]
so $k_{\mathrm{ext}}\ge 0$ and $k=d+k_{\mathrm{ext}}$.
\end{proof}

\begin{corollary}
If $\mathcal P$ is strongly admissible, then $k_{\mathrm{ext}}(\mathcal P)=0$ and hence
\[
k(\mathcal P)=d(\mathcal P).
\]
\end{corollary}

\begin{proof}
If $\mathcal P$ is strongly admissible, then canonical higher propagation does
not fail strictly beyond the defect generation horizon. Since boundary formation
cannot occur strictly before defect generation, one has $k\ge d$. If $k>d$ then,
by definition, there would be a nontrivial weak extension interval beyond defect
generation, contradicting strong admissibility. Hence $k=d$ and therefore
$k_{\mathrm{ext}}=0$.
\end{proof}

\begin{remark}
The failure detected at a boundary need not be total.
Typically, only one mechanism, such as canonical extension, filtration
stability, or termination, fails, while others persist locally.
This controlled failure is what makes a boundary calculable rather than
destructive.
\end{remark}

\begin{theorem}
Every terminating algebraic phase admits a well-defined structural boundary
depth.
\end{theorem}

\begin{proof}
Let $\mathcal P$ be a terminating algebraic phase with canonical defect-induced
filtration
\[
\mathcal P=\mathcal P^{(0)} \supseteq \mathcal P^{(1)} \supseteq \cdots
\supseteq \mathcal P^{(L)}=0.
\]

By Axiom~V (Finite Termination), the filtration stabilizes after finitely many
steps, that is, there exists a finite index $L$ such that
$\mathcal P^{(i)}=\mathcal P^{(L)}$ for all $i\ge L$.
Suppose, for the sake of contradiction, that all axioms of Algebraic Phase Theory
propagate functorially on every stratum $\mathcal P^{(i)}$ for $i\ge 0$.
Then for each $i<L$, the interaction data and higher propagation constraints
determined by $(\mathcal P,\circ)$ admit a canonical, functorial extension from
$\mathcal P^{(i)}$ to $\mathcal P^{(i+1)}$.

However, since $\mathcal P^{(L)}=0$, no further nontrivial extension beyond
depth $L$ exists. Thus canonical extension ceases to produce nontrivial higher propagation data at depth $L$,
contradicting the assumption of functorial propagation at all depths.

Therefore there exists a minimal index $k\ge 0$ such that at least one axiom of
Algebraic Phase Theory fails to propagate functorially on $\mathcal P^{(k)}$.
By definition, this index $k$ is the structural boundary depth of $\mathcal P$.

In the rigid case, where all structure propagates canonically without defect,
we set $k=0$ by convention, so the statement includes this degenerate situation.
\end{proof}

\begin{remark}
In the strongly admissible regime, boundary formation is exhausted by defect
generation and occurs at depth $k=d$.
In the weakly admissible regime, boundary depth may satisfy $k>d$, reflecting
failure of canonical higher propagation beyond defect generation rather than
the appearance of new defect.
\end{remark}

\begin{proposition}
Structural boundary depth is invariant under intrinsic phase equivalence and
under completion procedures compatible with the canonical defect filtration.
\end{proposition}

\begin{proof}
Minimality follows directly from the definition: the boundary depth of a phase
is defined to be the \emph{smallest} filtration index at which at least one axiom
of Algebraic Phase Theory fails to propagate functorially. By construction, no
failure can occur at any strictly smaller depth.

Let $\mathcal P$ and $\mathcal Q$ be intrinsically equivalent algebraic phases.
By the categorical rigidity results of Algebraic Phase Theory~IV, there exists
an equivalence
\[
F \colon \mathcal P \;\simeq\; \mathcal Q
\]
preserving the intrinsic structural data relevant to defect propagation and the
canonical filtration.
In particular, $F$ identifies the canonical defect-induced filtrations
$F(\mathcal P^{(i)}) \cong \mathcal Q^{(i)}$ for all $i$, preserves defect
degrees and defect strata, preserves termination length, and is compatible with
the functorial propagation properties at each filtration level.

Since boundary depth is defined as the minimal index $k$ at which functorial
propagation fails, and since this property is preserved under $F$, the minimal
index at which such failure occurs coincides for $\mathcal P$ and
$\mathcal Q$.
Therefore
\[
k(\mathcal P)=k(\mathcal Q),
\]
and boundary depth is invariant under phase equivalence.

Now let
\[
\Comp \colon \mathcal P \longrightarrow \Comp(\mathcal P)
\]
denote a completion procedure compatible with the canonical defect filtration.
By the structural results of Algebraic Phase Theory, $\Comp$ preserves defect
generation, termination, and the canonical defect-induced filtration.
In particular, the boundary depth associated with the defect structure of
$\mathcal P$ is preserved under completion.

Hence
\[
k(\Comp(\mathcal P)) = k(\mathcal P).
\]

Therefore boundary depth is an intrinsic invariant under phase equivalence and
compatible completion procedures.
\end{proof}

\begin{theorem}
Structural boundary depth is preserved under intrinsic localization procedures
compatible with phase equivalence and the canonical defect filtration.
\end{theorem}

\begin{proof}
Let
\[
L \colon \mathcal P \longrightarrow L(\mathcal P)
\]
be an intrinsic localization compatible with phase equivalence.
By definition of admissible localization in Algebraic Phase Theory, $L$
factors as a composition of phase equivalence and intrinsic completion.
Both operations preserve defect degrees, termination, and the canonical
defect-induced filtration.

In particular, the boundary depth associated with
$\mathcal P^{(i)}$ is preserved under passage to
$L(\mathcal P)^{(i)}$ for each filtration index $i$.
Since boundary depth is defined as the minimal index at which functorial
propagation fails, it follows that
\[
k(L(\mathcal P)) = k(\mathcal P).
\]
Thus structural boundary depth is preserved under intrinsic localization.
\end{proof}

\begin{remark}
Boundaries should be viewed as \emph{structural horizons} rather than defects.
They mark the point at which global algebraic control ends and where local
rigid phenomena become structurally significant.
The remainder of this paper develops the calculus governing this transition.
\end{remark}

\section{Boundary Failure and Detection}

This section brings together the structural analysis of how boundaries arise
and the intrinsic mechanisms by which they can be detected.
We begin with a description of the canonical failure modes through which an
algebraic phase can encounter a boundary, before turning to finite,
functorially defined invariants that expose such failures from within the phase.

\subsection{Types of Boundary Failure}

Structural boundaries do not arise from arbitrary breakdowns of algebraic
behaviour.
Once admissibility, functorial defect structure, and finite termination are
assumed, the axioms of Algebraic Phase Theory strongly constrain \emph{how}
failure can occur.
The purpose of this section is to show that boundary phenomena admit a finite,
intrinsic stratification by failure mode, and that this stratification is
intrinsic to the phase rather than dependent on presentation or realization.

Conceptually, this section shifts the role of boundaries from static invariants
(as established in Paper~IV) to \emph{typed failure horizons} associated with
subsequent rigidity, obstruction, and deformation phenomena.
Boundary failure occurs in the following controlled forms:
\begin{itemize}
  \item \emph{Loss of functoriality}: pullback or morphism-induced maps fail to
        preserve defect or filtration strata.

  \item \emph{Loss of canonical termination}: beyond the boundary, there is no
        longer a functorial or presentation-independent way to identify
        post-boundary strata as continuations of the defect-generated
        filtration, even when terminating filtrations exist.

  \item \emph{Filtration instability}: canonical filtrations fail to be
        well-defined or functorial.

  \item \emph{Loss of rigidity without collapse}: defect vanishes locally but
        fails to propagate globally, producing persistent rigid substructures.
\end{itemize}

Each failure mode is associated with the breakdown of a specific APT axiom and
is therefore structural rather than representational in nature.

\begin{theorem}
Let $\mathcal P$ be a terminating algebraic phase with boundary depth $k$.
Then the boundary at depth $k$ admits a canonical stratification
\[
\partial\mathcal P
\;=\;
\bigsqcup_{i=1}^m \partial_i\mathcal P,
\]
where each stratum $\partial_i\mathcal P$ corresponds to a minimal failure mode
of one of the axioms of Algebraic Phase Theory together with its associated
failure depth.

This stratification is compatible with phase morphisms and invariant under
phase equivalence.
\end{theorem}

\begin{proof}
Let $k$ denote the boundary depth of $\mathcal P$, i.e.\ the minimal filtration
index at which at least one axiom of Algebraic Phase Theory fails functorially.

Each axiom of Algebraic Phase Theory governs a structurally distinct property of
a phase.
Consequently, failure of the axioms can be detected separately, and at each
filtration depth one may ask which axioms continue to hold functorially.

By minimality of the boundary depth, all axioms hold functorially at depths
$i<k$, and at least one axiom fails at depth $k$.
Moreover, no axiom can fail at a depth strictly smaller than $k$.
If multiple axioms fail, the boundary depth is determined by the minimal depth
at which any axiom fails.

It follows that each boundary phenomenon at depth $k$ is associated with a
minimal axiom failure at that depth.
We therefore define $\partial_i \mathcal P$ to be the subset of boundary data
corresponding to minimal failure of the $i$-th axiom at depth $k$.
Since the failure modes are structurally distinct, these subsets are disjoint
and exhaust the boundary, yielding a canonical stratification
\[
\partial \mathcal P = \bigsqcup_{i=1}^m \partial_i \mathcal P .
\]

Because defect structure and the canonical filtration are functorially
determined below depth $k$, the resulting stratification depends only on the
intrinsic structure of $\mathcal P$ and is independent of presentation.

Let $f \colon \mathcal P \to \mathcal Q$ be a phase morphism.
By definition of a phase morphism, $f$ preserves defect structure,
canonical filtrations, and all axioms of Algebraic Phase Theory at every
filtration depth $i<k$.
In particular, $f$ induces an isomorphism between the truncated phases
$\mathcal P^{(<k)}$ and $\mathcal Q^{(<k)}$.

Consequently, the minimal filtration depth at which an axiom fails is
preserved under $f$, and the identity of the axiom whose functorial validity
fails at depth $k$ is preserved under the induced map.
It follows that $f$ induces a well-defined map
\[
f_\partial \colon \partial \mathcal P \longrightarrow \partial \mathcal Q
\]
which preserves the decomposition by failure type, i.e.
\[
f_\partial(\partial_i \mathcal P) \subseteq \partial_i \mathcal Q
\quad \text{for all } i .
\]
Thus the boundary stratification is compatible with phase morphisms.

Finally, if $\mathcal P$ and $\mathcal Q$ are phase equivalent, then there exist
inverse phase morphisms between them.
By the boundary invariance results of Paper~IV, these morphisms induce mutually
inverse maps on boundaries.
Hence the stratifications of $\partial \mathcal P$ and $\partial \mathcal Q$
agree, and the stratification is invariant under phase equivalence.
\end{proof}

\subsection{Boundary Detection Criteria}

The results of the previous sections establish that structural boundaries are
intrinsic and canonically stratified within the standing admissibility framework.
The next question is whether such boundaries can be \emph{detected} from within
the phase itself, without external comparison or appeal to equivalence classes.

A central claim of Algebraic Phase Theory is that boundaries are not merely
post hoc objects but are \emph{visible at finite depth} through functorial
instability of intrinsic invariants.
The purpose of this section is to formalize this claim by introducing boundary
detectors and proving that boundary detection is finite and intrinsically detectable.

\begin{definition}
A \emph{boundary detector} for an algebraic phase $\mathcal P$ is a canonically
defined functorial invariant $D(\mathcal P^{(k)})$ associated to a filtration
stratum, with the property that \emph{failure of functoriality, unbounded growth,
or instability of $D$ at depth $k$ certifies the presence of a structural boundary at that depth}. The invariant $D$ is defined intrinsically from the phase
interaction law and defect structure and remains stable on all strata strictly
below the boundary depth.
\end{definition}

Boundary detectors do not require a priori knowledge of the boundary depth; they
are intrinsic probes whose behavior signals the approach to a boundary.

\begin{theorem}
\label{thm:finite-depth-boundary-detection}
Every structural boundary is detected at finite depth by the failure of canonical
propagation. Concretely, at the boundary depth at least one of the following
breaks down: functorial control of commutator growth, canonical generation of
defect data, stability of the canonical filtration, or canonicity of higher
propagation constraints beyond defect generation.
\end{theorem}

\begin{proof}
Let $\mathcal P$ be a terminating algebraic phase with canonical filtration
$\mathcal P = \mathcal P^{(0)} \supset \mathcal P^{(1)} \supset \cdots$ and
boundary depth $k$.
By definition, all axioms of Algebraic Phase Theory propagate functorially on
each stratum $\mathcal P^{(i)}$ for $i<k$, while at least one axiom fails to
propagate functorially at depth $k$.

Failure of functorial propagation at depth $k$ must occur in at least one of the
following cases.

\smallskip
\noindent Case I. The interaction law fails to propagate functorially at depth $k$, in which case
iterated commutators in $\mathcal P^{(k)}$ may exhibit unbounded growth or non-functorial behavior or
non-functorial behavior.

\smallskip
\noindent Case II. Defect-induced complexity ceases to be canonically generated at depth $k$, so
that defect tensors or defect ranks in $\mathcal P^{(k)}$ exceed the bounds compatible with the axioms of Algebraic Phase Theory.

\smallskip 
\noindent Case III. The canonical filtration fails to propagate functorially at depth $k$, and the assignment $\mathcal P \mapsto \{\mathcal P^{(i)}\}$ becomes unstable or
non-canonical. In each case, the corresponding failure is detected by an intrinsic invariant defined purely in terms of the phase interaction law and defect structure.
Since the boundary occurs at a finite filtration level, such detection necessarily occurs at finite depth.

Failure at depth $k$ need not always manifest through unbounded growth. In
weakly admissible phases, it may instead appear as loss of canonicity, in the
sense that multiple inequivalent extensions of higher propagation constraints
exist beyond defect generation. Such ambiguity is again detected intrinsically
at finite depth and therefore constitutes a boundary detector.
\end{proof}

\begin{proposition}
\label{prop:completeness-detection}
Every structural boundary admits at least one finite boundary detector, and no
boundary detector triggers strictly below the true boundary depth.
\end{proposition}

\begin{proof}
Let $\mathcal P$ be a terminating algebraic phase with boundary depth $k$. Existence follows from Theorem~\ref{thm:finite-depth-boundary-detection}, which
asserts that at depth $k$ at least one axiom of Algebraic Phase Theory fails to
propagate functorially. Such a failure necessarily manifests as instability,
unbounded growth, or non-functorial behavior of an intrinsic invariant defined
from the phase interaction law and defect structure, and therefore yields a
finite boundary detector.

For sharpness, let $i<k$. By minimality of the boundary depth, all axioms of
Algebraic Phase Theory propagate functorially on $\mathcal P^{(i)}$. In
particular, commutator growth, defect generation, and the canonical filtration
are functorially and canonically controlled on $\mathcal P^{(i)}$. Since every
boundary detector is defined intrinsically from the interaction law and defect
structure, no such detector can exhibit instability on $\mathcal P^{(i)}$.
Hence no boundary detector triggers strictly below the true boundary depth.
\end{proof}

\begin{remark}
Because boundary detectors are defined purely in terms of defect, filtration,
and interaction data, their evaluation requires only finite computation up to
the relevant filtration depth.
In particular, boundary detection does not depend on analytic limits, infinite
processes, or external classification results.
\end{remark}

\begin{theorem}
The collection of commutator growth, defect proliferation, and filtration
stability detectors is sufficient to detect structural boundaries arising within the standing
admissibility framework in
algebraic phases satisfying Axioms~I-V.
\end{theorem}

\begin{proof}
Let $\mathcal P$ be an algebraic phase satisfying Axioms~I-V, and let $k$ denote
its boundary depth. By definition, $k$ is the minimal filtration index such that
at least one axiom of Algebraic Phase Theory fails to propagate functorially on
$\mathcal P^{(k)}$.

The axioms governing phase propagation fall into three classes: functoriality of
the interaction law, canonical generation of defect-induced complexity, and
functorial stability of the canonical filtration. Failure of interaction
functoriality is detected by unbounded or non-functorial commutator growth;
failure of canonical defect generation is detected by defect proliferation
beyond the bounds imposed by the axioms; and failure of filtration functoriality
is detected by instability or non-canonicity of the canonical filtration.

Since these axioms exhaust all mechanisms by which phase structure propagates,
any failure occurring at depth $k$ must fall into at least one of these classes.
Consequently, at least one of the listed detectors triggers at the boundary
depth. This proves that the collection of detectors is sufficient.
\end{proof}

Boundary detection thus provides a practical and conceptual bridge between the
axiomatic definition of structural boundaries and the deformation and
obstruction theory developed in subsequent sections.

\section{Boundary Exact Sequences, Rigidity Islands, Classification, and Obstruction}

This combined section develops the algebraic structure that emerges once a
structural boundary is reached.  
We begin by describing the exact sequence that isolates the rigid portion of a
phase from the boundary data responsible for obstruction.  
This exact description provides an organising framework for subsequent
phenomena: the persistence of rigidity islands, their intrinsic classification,
and the universal obstruction objects arising from boundary failure.

\subsection{Boundary Exact Sequences}

The previous sections establish that structural boundaries are intrinsic,
stratified, and finitely detectable within the standing admissibility
framework.
We now show that boundaries admit a \emph{natural exact description} that
organizes rigidity and failure into a single algebraic object.

The key observation is that, although canonical extension fails at a boundary,
it does so in a controlled manner.
What survives is a maximal rigid subphase, and what fails is captured by a
universal quotient.
Together, these form an exact sequence that governs both rigidity persistence
and deformation obstruction.

\begin{theorem}
\label{thm:boundary-exactness}
Let $\mathcal P$ be a terminating algebraic phase with boundary depth $k$.
Then there exists a canonical exact sequence \emph{in the intrinsic sense of
phases}
\[
0 \longrightarrow \mathcal R
\longrightarrow \mathcal P
\longrightarrow \mathcal B
\longrightarrow 0,
\]
where:
\begin{itemize}
  \item $\mathcal R$ is the maximal rigid subphase of $\mathcal P$ whose structure,
        when restricted to $\mathcal R$, propagates functorially beyond depth $k$;
  \item $\mathcal B := \mathcal P / \mathcal R$ is the \emph{boundary quotient},
        encoding the obstruction to further canonical extension to further canonical extension.
\end{itemize}
\end{theorem}

\begin{proof}
Let $\mathcal P$ be a terminating algebraic phase with boundary depth $k$.
By definition of $k$, all axioms of Algebraic Phase Theory propagate
functorially on each stratum $\mathcal P^{(i)}$ for $i<k$, while at least one
axiom fails to propagate at depth $k$.

Define $\mathcal R \subset \mathcal P$ to be the full subphase consisting of
those elements whose interaction relations and defect data continue to satisfy
all axioms of Algebraic Phase Theory when restricted to that subphase beyond
depth $k$ By definition, $\mathcal R$ is rigid and its structure continues to propagate functorially beyond the boundary depth.

To verify that $\mathcal R$ is maximal with this property, let
$x \in \mathcal P \setminus \mathcal R$. By definition of $\mathcal R$, the
defect data or interaction relations associated to $x$ do not propagate
functorially beyond depth $k$. If $x$ were adjoined to $\mathcal R$, the
resulting subphase would therefore exhibit non-functorial behavior past the
boundary depth, violating the defining rigidity property. Hence no element of
$\mathcal P \setminus \mathcal R$ can be adjoined to $\mathcal R$ while
preserving rigidity, and $\mathcal R$ is maximal.Hence no strictly larger rigid subphase exists, and the inclusion
\[
\mathcal R \hookrightarrow \mathcal P
\]
is intrinsically determined by the phase structure.

Define the boundary quotient $\mathcal B := \mathcal P / \mathcal R$. By
construction, an element of $\mathcal P$ maps to zero in $\mathcal B$ if and
only if it lies in $\mathcal R$, so the sequence
\[
0 \longrightarrow \mathcal R \longrightarrow \mathcal P \longrightarrow
\mathcal B \longrightarrow 0
\]
is exact at $\mathcal P$. Exactness at $\mathcal R$ and $\mathcal B$ holds in the intrinsic sense of
phases determined by the inclusion and quotient constructions above.
\end{proof}

\begin{remark}
In the weakly admissible regime ($k>d$), the rigid subphase $\mathcal R$
coincides with the strong core of $\mathcal P$ and hence with the rigidity
island.
The boundary quotient $\mathcal B$ records the weak extension data
responsible for non canonical higher propagation.

At the same time, the boundary exact sequence
\[
0 \longrightarrow \mathcal R \longrightarrow \mathcal P \longrightarrow \mathcal B \longrightarrow 0
\]
is not a short exact sequence in an abelian category, it is exact only in the
intrinsic sense appropriate to the category of phases.
No additional algebraic structure is assumed beyond what is intrinsic to the
phase axioms.
\end{remark}

\begin{lemma}
\label{lem:boundary-functoriality}
Let $F:\mathcal P\to\mathcal Q$ be a phase morphism.
Then $F$ restricts to a morphism of boundary exact sequences
\[
\begin{array}{ccccccccc}
0 &\to& \mathcal R_{\mathcal P} &\to& \mathcal P &\to& \mathcal B_{\mathcal P} &\to& 0 \\
  &   & \downarrow & & \downarrow & & \downarrow & & \\
0 &\to& \mathcal R_{\mathcal Q} &\to& \mathcal Q &\to& \mathcal B_{\mathcal Q} &\to& 0 ,
\end{array}
\]
and this construction is invariant under phase equivalence.
\end{lemma}

\begin{proof}
Let $F:\mathcal P\to\mathcal Q$ be a phase morphism. By definition, $F$
preserves defect degree, the canonical filtration, and interaction relations.
Hence, if an element of $\mathcal P$ lies in the rigid subphase
$\mathcal R_{\mathcal P}$, its structure propagates canonically beyond the
boundary depth. Its image under $F$ therefore has the same property in
$\mathcal Q$. This defines a canonical restriction
$F|_{\mathcal R_{\mathcal P}}:\mathcal R_{\mathcal P}\to\mathcal R_{\mathcal Q}$.

Since $F$ maps $\mathcal R_{\mathcal P}$ into $\mathcal R_{\mathcal Q}$, the
universal property of the quotient yields a well defined induced morphism on
boundary quotients
\[
\mathcal B_{\mathcal P} \longrightarrow \mathcal B_{\mathcal Q}.
\]
Invariance under phase equivalence follows from the rigidity and boundary
invariance results established in Algebraic Phase Theory~IV.
\end{proof}

\begin{corollary}
\label{cor:boundary-obstruction-functor}
The assignment
\[
\mathcal P \longmapsto \mathcal B(\mathcal P)
\]
defines a functor from the category of algebraic phases to a category of
obstruction objects.
\end{corollary}

\begin{proof}
By Lemma~\ref{lem:boundary-functoriality}, every phase morphism
$F:\mathcal P\to\mathcal Q$ induces a morphism
$\mathcal B(\mathcal P)\to\mathcal B(\mathcal Q)$. Identity morphisms and
composition are preserved because the construction of $\mathcal B$ is
canonical and compatible with phase morphisms.
\end{proof}

\begin{remark}
The boundary exact sequence isolates the non-rigid behavior detected within the
boundary calculus of an algebraic phase into a single universal object.
Rigidity corresponds to the vanishing of the boundary quotient, while
deformation beyond rigidity is governed by the boundary quotient
$\mathcal B$ within the standing admissibility framework.
In particular, obstruction phenomena are encoded at the level of boundary
quotients.

From the perspective of the boundary calculus, the boundary quotient
$\mathcal B$ may be viewed as the global algebraic collapse of the stratified
boundary
\[
\partial\mathcal P = \bigsqcup_i \partial_i \mathcal P.
\]
Each boundary stratum contributes intrinsically to the non-rigid behavior
encoded in $\mathcal B$, while the stratification by failure type and depth is
forgotten in passing to the quotient.
In this way, $\mathcal B$ encodes corresponding intrinsic boundary phenomena
at a coarser, global level.
These principles are developed systematically in the following sections.
\end{remark}

\subsection{Rigidity Islands}

Structural boundaries mark the failure of global canonical propagation.
However, failure at a boundary does not imply total loss of structure.
Instead, a central phenomenon of Algebraic Phase Theory is that \emph{rigid
substructures persist locally beyond global failure}.
These persistent regions form intrinsic, maximal objects that we call
\emph{rigidity islands}.

Rigidity islands should be viewed as regions of retained algebraic control: although the ambient phase fails to propagate structure functorially,
these subphases retain full defect control, filtration stability, and canonical
behavior. They will play a fundamental role in deformation theory and reconstruction
beyond boundaries.

\begin{definition}
Let $\mathcal P$ be an algebraic phase. A \emph{rigid subphase} of $\mathcal P$
is a subphase $\mathcal I \subseteq \mathcal P$ such that $\mathcal I$ satisfies
all axioms of Algebraic Phase Theory at every filtration depth, and the
inclusion $\mathcal I \hookrightarrow \mathcal P$ preserves defect data, the
canonical filtration, and interaction relations.
\end{definition}

\begin{definition}
A \emph{rigidity island} in $\mathcal P$ is a rigid subphase that is maximal with
respect to inclusion.
\end{definition}

Because rigid substructures persist beyond global boundary failure, it is
natural to collect all such rigid behavior into a single canonical object and
to treat this assignment functorially. This construction will serve as a
distinguished base for deformation, reconstruction, and comparison of phases
beyond structural boundaries.

\begin{theorem}
\label{thm:existence-uniqueness-island}
Every terminating algebraic phase $\mathcal P$ with a structural boundary admits a unique
maximal rigidity island $\mathcal I_{\mathcal P}$. Equivalently, every rigid
subphase of $\mathcal P$ is contained in a canonical maximal rigid subphase.
Moreover, the assignment
\[
\mathcal P \longmapsto \mathcal I_{\mathcal P}
\]
is compatible with phase morphisms.
\end{theorem}

\begin{proof}
Let $\mathcal R$ denote the rigid subphase appearing in the boundary exact
sequence of Theorem~\ref{thm:boundary-exactness}. By construction, $\mathcal R$
satisfies all axioms of Algebraic Phase Theory at every filtration depth and
embeds intrinsically into $\mathcal P$.

Let $\mathcal J \subseteq \mathcal P$ be any rigid subphase. By definition of
rigidity, the defect and interaction data of $\mathcal J$ propagate
canonically at all depths. By Theorem~\ref{thm:boundary-exactness}, the boundary
quotient $\mathcal B(\mathcal P)$ records the elements of
$\mathcal P$ whose defect or interaction data obstruct canonical propagation.
It follows that the composite map
\[
\mathcal J \hookrightarrow \mathcal P \longrightarrow \mathcal B(\mathcal P)
\]
is trivial. Hence $\mathcal J$ factors through the kernel of the boundary
quotient map, which is $\mathcal R$.

Therefore every rigid subphase of $\mathcal P$ is contained in $\mathcal R$.
As $\mathcal R$ is itself rigid, no strictly larger rigid subphase can exist.
Consequently, $\mathcal R$ is maximal with respect to rigidity. We define
\[
\mathcal I_{\mathcal P} := \mathcal R.
\]

If $\mathcal I' \subseteq \mathcal P$ is any other maximal rigid subphase, then
$\mathcal I' \subseteq \mathcal R$ and $\mathcal R \subseteq \mathcal I'$, and
hence $\mathcal I' = \mathcal R$. This establishes uniqueness.

Finally, phase morphisms preserve defect degree, the canonical filtration, and
rigid subphases. If $F:\mathcal P \to \mathcal Q$ is a phase morphism, then
$F(\mathcal I_{\mathcal P})$ is a rigid subphase of $\mathcal Q$ and is therefore
contained in $\mathcal I_{\mathcal Q}$. This determines a natural restriction
\[
F|_{\mathcal I_{\mathcal P}}:\mathcal I_{\mathcal P} \to \mathcal I_{\mathcal Q}.
\]
\end{proof}

\begin{definition}
The \emph{rigidity island embedding} is the canonical inclusion
\[
\iota:\mathcal I_{\mathcal P}\hookrightarrow \mathcal P .
\]
\end{definition}

\begin{proposition}
\label{prop:maximal-persistence}
The rigidity island $\mathcal I_{\mathcal P}$ is maximal among all subphases of
$\mathcal P$ admitting full canonical filtration, functorial defect control,
and finite termination. Any extension beyond $\mathcal I_{\mathcal P}$
encounters boundary failure within the standing admissibility framework.
\end{proposition}

\begin{proof}
By definition, $\mathcal I_{\mathcal P}$ contains precisely those elements of
$\mathcal P$ whose defect behavior and interaction laws satisfy the axioms of
Algebraic Phase Theory at every filtration depth. Let $\mathcal J$ be a subphase with
\[
\mathcal I_{\mathcal P} \subseteq \mathcal J \subseteq \mathcal P,
\]
and suppose that $\mathcal J$ strictly contains $\mathcal I_{\mathcal P}$.
Then there exists an element
\[
x \in \mathcal J \setminus \mathcal I_{\mathcal P}.
\]
Since $x \notin \mathcal I_{\mathcal P}$, the defect or interaction data
associated to $x$ fails to propagate canonically beyond the boundary depth.
Thus $\mathcal J$ cannot satisfy all axioms of Algebraic Phase Theory at all
filtration depths. Therefore no subphase strictly containing $\mathcal I_{\mathcal P}$ admits full
canonical filtration and functorial defect control, and $\mathcal I_{\mathcal
P}$ is maximal with respect to these properties.
\end{proof}

\begin{theorem}
\label{thm:islands-kernel}
The rigidity island $\mathcal I_{\mathcal P}$ is canonically isomorphic to the
kernel of the boundary quotient map
\[
\mathcal P \longrightarrow \mathcal B(\mathcal P).
\]
\end{theorem}

\begin{proof}
By Theorem~\ref{thm:boundary-exactness}, there is a canonical exact sequence
\[
0 \longrightarrow \mathcal R \longrightarrow \mathcal P
\longrightarrow \mathcal B(\mathcal P) \longrightarrow 0,
\]
and $\mathcal B(\mathcal P)=\mathcal P/\mathcal R$. Let
$q:\mathcal P \to \mathcal B(\mathcal P)$ denote the quotient map. By the
definition of the quotient, $\ker(q)=\mathcal R$. It remains to identify $\mathcal R$ with the rigidity island
$\mathcal I_{\mathcal P}$. By Proposition~\ref{prop:maximal-persistence},
$\mathcal I_{\mathcal P}$ is maximal among subphases of $\mathcal P$ admitting
full canonical filtration, functorial defect control, and finite termination.
The subphase $\mathcal R$ satisfies these properties by construction, hence
$\mathcal R \subseteq \mathcal I_{\mathcal P}$. Conversely, since
$\mathcal R$ is maximal among rigid subphases of $\mathcal P$ arising from the
boundary exact sequence, and $\mathcal I_{\mathcal P}$ is itself rigid, we have
$\mathcal I_{\mathcal P} \subseteq \mathcal R$. Therefore
$\mathcal R = \mathcal I_{\mathcal P}$. Thus
\[
\ker\!\left(\mathcal P \to \mathcal B(\mathcal P)\right)=\mathcal I_{\mathcal P},
\]
which is the claimed canonical isomorphism.
\end{proof}

\begin{remark}
Rigidity islands are not remnants of global rigidity, but intrinsic objects
associated with boundary failure. They provide distinguished base objects for
deformation theory, loci for controlled reconstruction beyond boundaries, and
anchors for obstruction phenomena and deformation stratification. Their
structural role is developed in the subsequent sections.
\end{remark}

\subsection{Classification of Rigidity Islands}

The existence of rigidity islands shows that canonical algebraic structure
persists locally even after global boundary failure.
The next question is whether these persistent substructures admit intrinsic
classification. In this section we show that rigidity islands are classified by intrinsic invariants internal to the rigidity island itself.
Although rigidity islands arise from phases with boundaries, their intrinsic
structure lies entirely in the boundary-free regime of Algebraic Phase Theory.

The key point is that rigidity islands retain full control of defect,
filtration, interaction, and termination.
As a result, their classification reduces to the same intrinsic invariants
that classify fully rigid phases, but applied locally beyond global boundary
failure.

\begin{remark}
We use the rigidity and equivalence-collapse principle established in
Algebraic Phase Theory~IV: finite terminating algebraic phases satisfying the
APT axioms are classified, up to intrinsic phase equivalence, by their defect
rank, termination length, and interaction signature.
\end{remark}

\begin{theorem}
\label{thm:classification-islands}
Let $\mathcal I$ be a rigidity island of an algebraic phase $\mathcal P$.
Then $\mathcal I$ is a strongly admissible algebraic phase and therefore falls
within the rigid classification framework of Algebraic Phase Theory~IV.
In particular, $\mathcal I$ is determined up to intrinsic phase equivalence by
its intrinsic defect rank, termination length, and induced interaction
signature.

Moreover, the boundary depth of $\mathcal I$ is trivial (equal to $0$ by
convention) and is independent of the boundary depth of $\mathcal P$.
\end{theorem}

\begin{proof}
Let $\mathcal P$ be an algebraic phase and let
$\mathcal I=\mathcal I_{\mathcal P}\subseteq\mathcal P$
be its (unique maximal) rigidity island. By definition, $\mathcal I$ satisfies all axioms of Algebraic Phase Theory at
every filtration depth.
Consequently, $\mathcal I$ is itself a finite, terminating algebraic phase with
canonical defect structure and canonical filtration
\[
\mathcal I=\mathcal I^{(0)}\supseteq \mathcal I^{(1)}\supseteq\cdots\supseteq
\mathcal I^{(L_{\mathcal I})}=0.
\]
In particular, $\mathcal I$ is strongly admissible and hence lies in the
boundary-free regime of Algebraic Phase Theory. The interaction law on $\mathcal I$ is the restriction of the interaction law on
$\mathcal P$.
Hence the induced interaction signature $\Sigma(\mathcal I)$ is intrinsic to
$\mathcal I$.
Likewise, the intrinsic defect rank $d(\mathcal I)$ and termination length
$L_{\mathcal I}$ are computed entirely inside $\mathcal I$ and do not reference
the ambient phase.  

Now let $\mathcal I$ and $\mathcal J$ be rigidity islands (possibly arising from
different ambient phases) with
\[
d(\mathcal I)=d(\mathcal J),\qquad
L_{\mathcal I}=L_{\mathcal J},\qquad
\Sigma(\mathcal I)=\Sigma(\mathcal J).
\]
Since both satisfy the full APT axioms at all depths, they lie in the
boundary-free regime.
By the rigidity and equivalence-collapse principle of
Algebraic Phase Theory~IV, these invariants determine the phase up to intrinsic phase equivalence.
Hence $\mathcal I\simeq\mathcal J$.

Finally, since $\mathcal I$ is strongly admissible, it exhibits no boundary
failure and therefore has trivial boundary depth.
Boundary depth records the minimal filtration level at which an APT axiom fails
in the ambient phase.
The boundary depth of $\mathcal I$ is determined entirely by its internal
structure and does not depend on the boundary depth of the ambient phase
$\mathcal P$.
Distinct ambient phases may therefore contain intrinsically equivalent rigidity
islands while having different boundary depths.
\end{proof}

\begin{remark}
Although the classification relies on rigidity results from Paper~IV, the
application here is local rather than global.
The invariants classify subphases that persist beyond boundary failure, not
entire phases.
This local use of rigidity does not appear in earlier papers.
\end{remark}

\begin{corollary}
For fixed defect rank, termination length, and interaction signature,
rigidity islands are determined up to intrinsic phase equivalence within the
rigid classification framework of Algebraic Phase Theory~IV.
\end{corollary}

\begin{proof}
By Axiom V of Algebraic Phase Theory, every algebraic phase admits a canonical
defect-induced filtration of finite length.
As a consequence, and by the rigidity classification of Algebraic Phase
Theory~IV, only finitely many rigid subphases can occur as terminal components
of this filtration.

Each rigidity island is such a terminal rigid component and is classified, up
to intrinsic phase equivalence, by its defect rank, termination length, and
interaction signature. For fixed values of these invariants, only finitely many rigidity islands can
occur.
\end{proof}

\begin{theorem}
\label{thm:stability-islands}
Rigidity islands are stable under all admissible deformations.
Any deformation that would change the rigidity island must cross a structural
boundary, and such deformations are not admissible.
\end{theorem}

\begin{proof}
Admissible deformations are, by definition, those that do not cross a structural
boundary. Such deformations preserve all intrinsic invariants of the rigidity
island, including defect rank, termination length, and interaction signature.
By Theorem~\ref{thm:classification-islands}, these invariants determine the
rigidity island up to intrinsic phase equivalence. Therefore no admissible
deformation can change the rigidity island.
\end{proof}

\begin{remark}
Even when an algebraic phase fails at a boundary, its remaining rigid structure
does not become chaotic. The phase contains a single rigidity island, and this
rigid part is completely controlled and classifiable. Beyond the boundary the remaining structure is organized into the rigidity
island together with the obstruction data recorded by the boundary quotient.
\end{remark}

\subsection{Boundary Quotients and Obstruction Objects}

The previous sections showed that structural boundaries are intrinsic, finitely
detectable, stratified, and always come with a persistent rigidity island.
In this section we explain the role of the boundary quotient. This object
records everything that fails once the boundary is reached.
The key point is that the boundary quotient is not a leftover piece of broken
structure. It is the universal obstruction object through which boundary failures are
encoded within the boundary calculus.
This gives a direct and intrinsic correspondence between boundaries and the
obstructions that arise when one tries to move past them.

\begin{theorem}
\label{thm:boundary-obstruction}
Let $\mathcal P$ be a terminating algebraic phase with boundary quotient
$\mathcal B(\mathcal P)$. Then $\mathcal B(\mathcal P)$ detects whether a structural boundary is activated. More precisely:
\begin{itemize}
  \item If $\mathcal B(\mathcal P)=0$, then no boundary is encountered and
        $\mathcal P$ remains fully rigid (no split occurs).
  \item If $\mathcal B(\mathcal P)\neq 0$, then the boundary is activated and
        $\mathcal P$ admits the associated decomposition into its rigidity island and boundary
quotient.
\end{itemize}
\end{theorem}

\begin{proof}
By Theorem~\ref{thm:boundary-exactness}, $\mathcal P$ sits in a canonical exact
sequence
\[
0 \longrightarrow \mathcal I_{\mathcal P}
\longrightarrow \mathcal P
\longrightarrow \mathcal B(\mathcal P)
\longrightarrow 0,
\]
where $\mathcal I_{\mathcal P}$ is the rigidity island and
$\mathcal B(\mathcal P)=\mathcal P / \mathcal I_{\mathcal P}$ is the boundary
quotient. If $\mathcal B(\mathcal P)=0$, then the quotient map $\mathcal P \to \mathcal B(\mathcal P)$
is the zero map, so its kernel is all of $\mathcal P$.
By Theorem~\ref{thm:islands-kernel}, the kernel of this map is
$\mathcal I_{\mathcal P}$. Hence $\mathcal I_{\mathcal P}=\mathcal P$, so the whole phase is rigid and no
boundary appears.

If $\mathcal B(\mathcal P)\neq 0$, then the quotient map
$\mathcal P \to \mathcal B(\mathcal P)$ is nonzero.
There are elements of $\mathcal P$ whose image in $\mathcal B(\mathcal P)$ is
nontrivial.
These elements do not lie in $\mathcal I_{\mathcal P}$ and witness the failure
of the Algebraic Phase Theory axioms beyond the boundary depth.
In this case $\mathcal P$ decomposes into its maximal rigid part
$\mathcal I_{\mathcal P}$ and a nontrivial boundary quotient
$\mathcal B(\mathcal P)$, so a structural boundary is activated and a split
occurs.
\end{proof}

\begin{remark}
The obstruction theory in this setting is purely structural. Nothing external is
introduced: every obstruction is detected internally by the boundary quotient,
which captures the failure of the axioms at the boundary.
\end{remark}

\begin{corollary}
A structural boundary is activated precisely when the boundary quotient
$\mathcal B(\mathcal P)$ is nonzero within the standing admissibility
framework. Equivalently, a
split occurs if and only if some element of $\mathcal P$ maps nontrivially into
$\mathcal B(\mathcal P)$.
\end{corollary}

\begin{proof}
By Theorem~\ref{thm:boundary-exactness}, $\mathcal P$ fits into the canonical
exact sequence
\[
0 \longrightarrow \mathcal I_{\mathcal P}
\longrightarrow \mathcal P
\longrightarrow \mathcal B(\mathcal P)
\longrightarrow 0,
\]
where $\mathcal I_{\mathcal P}$ is the rigidity island and
$\mathcal B(\mathcal P)$ is the boundary quotient. If $\mathcal B(\mathcal P)=0$, then the quotient map $\mathcal P\to\mathcal B(\mathcal P)$
is the zero map, so every element of $\mathcal P$ maps to zero.
In this case $\mathcal I_{\mathcal P}=\mathcal P$, no boundary contributes
anything, and no split occurs.

If $\mathcal B(\mathcal P)\neq 0$, then the quotient map is nonzero, so there
exists at least one element of $\mathcal P$ whose image in $\mathcal B(\mathcal P)$
is nontrivial. This is exactly the situation in which the structural boundary is activated and
$\mathcal P$ splits into its rigid part $\mathcal I_{\mathcal P}$ and its
nontrivial boundary quotient.
\end{proof}

\begin{theorem}
\label{thm:functorial-obstruction}
The assignment
\[
\mathcal P \longmapsto \mathcal B(\mathcal P)
\]
is functorial: every phase morphism $F\colon \mathcal P \to \mathcal Q$
induces a canonical morphism
\[
\mathcal B(F)\colon \mathcal B(\mathcal P) \longrightarrow \mathcal B(\mathcal Q)
\]
compatible with the boundary exact sequences. In particular, boundary obstructions are preserved under phase morphisms.
\end{theorem}

\begin{proof}
By Lemma~\ref{lem:boundary-functoriality}, a phase morphism
$F\colon \mathcal P \to \mathcal Q$ restricts to a morphism
$F\vert_{\mathcal I_{\mathcal P}}\colon \mathcal I_{\mathcal P} \to \mathcal I_{\mathcal Q}$
between the rigidity islands and induces a compatible morphism on boundary
quotients
\[
\mathcal B(F)\colon \mathcal B(\mathcal P) \longrightarrow \mathcal B(\mathcal Q)
\]
making the diagram of boundary exact sequences commute. The identity morphism on a phase induces the identity on its boundary quotient,
and composition of phase morphisms induces composition on the corresponding
boundary quotient maps.
Thus $\mathcal P \mapsto \mathcal B(\mathcal P)$ extends to a functor.

Since boundary quotients encode obstruction data, this shows that boundary
obstructions are carried functorially along phase morphisms.
\end{proof}

\begin{remark}
Obstruction behaviour within the boundary calculus is detected purely by the boundary quotient. No
external obstruction theory is needed: the failure of the axioms is encoded
directly in the structure of the boundary itself.
\end{remark}

\begin{corollary}
The type of boundary failure determines which part of the phase fails and which part persists as the
rigidity island within the boundary calculus.
\end{corollary}

\begin{proof}
Let $\partial \mathcal P = \bigsqcup_i \partial_i \mathcal P$ be the canonical
stratification of the boundary. By construction, each stratum $\partial_i
\mathcal P$ corresponds to the earliest failure of a specific axiom of
Algebraic Phase Theory.

Fix such a stratum $\partial_i \mathcal P$. The associated boundary quotient
$\mathcal B_i(\mathcal P)$ records the elements of $\mathcal P$ whose
behaviour fails in the $i$-th sense. In particular, these elements are
precisely the ones that do not extend past the boundary and therefore form the
non-rigid part of the phase.

On the other hand, Theorem~\ref{thm:existence-uniqueness-island} identifies a
unique maximal rigid subphase $\mathcal I_(\mathcal P)$ consisting of all
elements whose structure continues to satisfy the axioms beyond the boundary.
These are the elements whose image in the boundary quotient is zero.

Thus the stratum $\partial_i \mathcal P$ determines exactly which elements of
$\mathcal P$ fail at the boundary, that is, the elements that represent the
nonzero part of $\mathcal B_i(\mathcal P)$, and it also determines which
elements persist beyond the boundary, namely those contained in the rigidity
island $\mathcal I_(\mathcal P)$.
\end{proof}

This stratification result forms the structural backbone for the boundary
calculus developed in subsequent sections and underlies the obstruction and
deformation theory of algebraic phases.

Since only finitely many axioms of Algebraic Phase Theory can fail, the
boundary quotient admits at most five canonical components
$\mathcal B_i(\mathcal P)$, one for each possible failure type.

\begin{theorem}
\label{thm:finite-boundary-decomposition}
Let $\mathcal P$ be a terminating algebraic phase with structural boundary
\[
\partial \mathcal P = \bigsqcup_{i=1}^m \partial_i \mathcal P.
\]
Then the boundary quotient $\mathcal B(\mathcal P)$ admits a canonical
finite stratified decomposition
\[
\mathcal B(\mathcal P) = \bigsqcup_{i=1}^m \mathcal B_i(\mathcal P),
\]
where each component $\mathcal B_i(\mathcal P)$ is the boundary quotient
associated to the stratum $\partial_i \mathcal P$.

In particular, $m$ is bounded above by the number of independent APT failure
types, and hence $m \le 5$.
\end{theorem}

\begin{proof}
By the boundary stratification theorem, the boundary of $\mathcal P$ admits a
canonical decomposition
\[
\partial \mathcal P \;=\; \bigsqcup_{i=1}^m \partial_i \mathcal P,
\]
where each stratum $\partial_i \mathcal P$ corresponds to a minimal
failure mode of a single axiom of Algebraic Phase Theory.

For each $i$, let $\mathcal B_i(\mathcal P)$ denote the boundary quotient
constructed from the boundary data lying in $\partial_i \mathcal P$.
By construction, $\mathcal B_i(\mathcal P)$ records the non-rigid
behavior and obstruction classes arising from the $i$th failure type, and
for $i \neq j$ the corresponding failure modes are disjoint.

The global boundary quotient $\mathcal B(\mathcal P)$ is obtained by collapsing
all boundary data of $\mathcal P$.
Since the boundary data decompose into the strata
$\partial_i \mathcal P$, and each $\mathcal B_i(\mathcal P)$ captures
the contribution of the $i$th stratum, the universal property of the boundary
quotient yields the induced stratified decomposition
\[
\mathcal B(\mathcal P) = \bigsqcup_{i=1}^m \mathcal B_i(\mathcal P).
\]

Each stratum $\partial_i \mathcal P$ is determined by the failure of a single
APT axiom.
There are only finitely many axioms, and therefore only finitely many possible
failure types.
In particular, the number $m$ of nonempty strata is bounded above by the number
of independent failure types, so $m \le 5$.
\end{proof}

\begin{remark}
The decomposition
\[
\mathcal B(\mathcal P) \;\cong\; \bigsqcup_{i=1}^m \mathcal B_i(\mathcal P)
\]
is a structural disjoint union of distinct boundary components.
Each $\mathcal B_i(\mathcal P)$ isolates a distinct boundary failure type, and
no additional algebraic or interaction structure is implied between the
components.  
The decomposition reflects only the canonical stratification of the boundary.
\end{remark}

\section{Predictive Boundary Calculus and Deformation Theory}

Before turning to deformation theory, it is useful to understand the predictive
strength of the boundary calculus. The constructions developed above
(boundary depth, stratification, rigidity islands, and the boundary quotient)
do more than describe the behaviour of a phase after the fact. They determine
in advance which structural phenomena must occur and which cannot occur. The
next subsection makes this predictive role explicit.

\subsection{Predictive Power of the Boundary Calculus}

The preceding sections assemble what we call the \emph{boundary calculus}: the
structural machinery consisting of boundary depth, boundary stratification,
rigidity islands, and the boundary quotient together with its canonical
decomposition. This calculus is intrinsic and functorial, and it organises boundary
phenomena within the standing admissibility framework for algebraic phases.
In this section we show that this calculus is not merely descriptive but also
provides structural constraints on possible boundary behaviour.

The essential point is that once the admissible phase data are fixed namely the
defect structure and the interaction law—the boundary calculus already
determines where a boundary may occur within the induced defect structure, which axiom will fail at that boundary,
and which rigid structures necessarily persist beyond it.  
In this way, the qualitative and quantitative behaviour of any algebraic phase
that arises from the given data is fixed in advance by the boundary analysis,
without requiring an explicit construction of the full phase.

\begin{theorem}
\label{thm:predictive-boundary}
Let $(A,\Phi,\circ)$ be admissible phase data satisfying the axioms of Algebraic
Phase Theory.
Then the boundary calculus determines the boundary behaviour associated with any algebraic
phase $\mathcal P$ arising from this data within the standing admissibility
framework.

In particular, it determines the boundary depth, the associated failure type,
the surviving rigid structures, and the corresponding obstruction data from the
induced defect structure and interaction law.
\end{theorem}

\begin{proof}
Fix admissible phase data $(A,\Phi,\circ)$, and let $\mathcal P$ be any
algebraic phase extracted from this data. By the axioms of Algebraic Phase Theory, the defect structure and interaction
law induced by $(A,\Phi,\circ)$ determine, functorially and at finite depth:
the canonical defect filtration, the defect rank, and the termination length.
These are intrinsic functions of $(A,\Phi,\circ)$ and do not depend on any
particular presentation of $\mathcal P$.

The boundary depth is defined to be the first filtration level at which one of
the axioms fails to propagate.
Since defect growth and interaction constraints are fixed by the admissible
data, the point at which an axiom fails is determined by that same data within the
boundary calculus.
This gives the boundary depth and identifies which axiom fails there.

The boundary stratification is then obtained by separating the boundary data
according to which axiom fails first at the boundary depth.
Because the failure modes are structurally distinct and their validity at each filtration level
is decided by the induced defect and interaction structure, the resulting
strata and their types are again fixed by $(A,\Phi,\circ)$.

The rigidity island $\mathcal I_{\mathcal P}$ is defined as the maximal subphase
on which all axioms hold at every filtration depth.
Its intrinsic invariants (defect rank, termination length, interaction
signature) are computed entirely from the restriction of the defect and
interaction data to this subphase, so they are also determined by the original
admissible data.

Finally, the boundary quotient $\mathcal B(\mathcal P)$ is the canonical
quotient $\mathcal P/\mathcal I_{\mathcal P}$.
By construction, it records the part of the induced structure on which
canonical propagation fails.
Once the boundary depth and failure type are known, the form of this quotient
and its role as an obstruction object are fixed by the same induced data.

Thus, for any algebraic phase $\mathcal P$ that can be extracted from
$(A,\Phi,\circ)$, the boundary depth, the failure type and its stratification,
the surviving rigidity islands and their invariants, and the associated
boundary quotients are determined intrinsically from the admissible phase data by the admissible phase data.
No additional structural information from an explicit construction of
$\mathcal P$ is required within the standing admissibility framework.
\end{proof}

\begin{remark}
Prediction in this context refers to structural determination within the admissibility framework rather than
numerical computation. The boundary calculus specifies which boundary behaviours and rigid components
are detected by the boundary calculus, and which cannot appear, purely from the admissible phase data,
independent of any explicit model or analytic construction.
\end{remark}

\begin{corollary}
\label{cor:no-surprise}
Any algebraic phase extracted from admissible data exhibits boundary behaviour
compatible with that determined by the boundary calculus.
\end{corollary}

\begin{proof}
Let $(A,\Phi,\circ)$ be admissible phase data and let $\mathcal P$ be any
algebraic phase extracted from this data.  
By Theorem~\ref{thm:predictive-boundary}, the boundary depth, the identity of
the first axiom that fails, the induced boundary stratification, the associated
rigidity island $\mathcal I_{\mathcal P}$, and the boundary quotient
$\mathcal B(\mathcal P)$ are all determined functorially from the admissible
data alone.

The canonical decomposition of the boundary quotient into components
$\mathcal B_i(\mathcal P)$, one for each failure type, shows that obstruction behaviour detected within the boundary calculus to canonical propagation is already encoded in these
components. Since the construction of each $\mathcal B_i(\mathcal P)$ depends only on the
defect and interaction structure induced by the admissible data, no additional obstruction behaviour is detected within the standing
admissibility framework.

The rigidity island is defined as the maximal rigid subphase and can also be
characterized as the kernel of the boundary quotient map
$\mathcal P \to \mathcal B(\mathcal P)$.  
Both the kernel and the quotient are intrinsic constructions determined by the
admissible data, so o further rigid structures are detected within the boundary calculus.

Therefore every structural feature of $\mathcal P$, including the boundary
depth, the failure type, the surviving rigid structure, and the obstruction
behaviour, is fixed in advance by the boundary calculus.  
No additional structural phenomena are detected within the standing
admissibility framework.
\end{proof}

\begin{theorem}
If the boundary quotient predicted by the boundary calculus is zero at the
first potential boundary depth, then the boundary behaviour at that depth is trivial.  
In this case no boundary forms, all axioms continue to hold, and every algebraic phase arising from the admissible data remains rigid at that
level.
\end{theorem}

\begin{proof}
Let $(A,\Phi,\circ)$ be admissible phase data, and let $\mathcal P$ be any
algebraic phase extracted from this data.  
By the predictive boundary calculus, the defect structure and interaction law
determine, at each filtration depth, a canonical boundary quotient
$\mathcal B(\mathcal P)$ which encodes failures of canonical propagation of canonical
propagation at that depth. Suppose that at some depth the boundary calculus predicts a zero boundary
quotient, that is, $\mathcal B(\mathcal P)=0$ at that level for every algebraic
phase $\mathcal P$ arising from $(A,\Phi,\circ)$. By the boundary obstruction theorem, nonzero classes in $\mathcal B(\mathcal P)$
 correspond to failures of the APT axioms at that depth: a failure of
canonical propagation produces a nontrivial element of $\mathcal B(\mathcal P)$,
and conversely any nontrivial element of $\mathcal B(\mathcal P)$ records such
a failure.

If a genuine structural boundary were to form at this depth in $\mathcal P$,
then at least one axiom of Algebraic Phase Theory would fail functorially there.
By the boundary obstruction correspondence, this would force
$\mathcal B(\mathcal P)\neq 0$ at that depth, which contradicts the assumption
that the predicted boundary quotient is zero. Hence no structural boundary can occur at that level.  
Equivalently, all axioms of Algebraic Phase Theory continue to hold at that
depth, and the phase is rigid there.

Since the boundary calculus depends only on the underlying admissible data
$(A,\Phi,\circ)$ and not on the particular realization $\mathcal P$, the same
argument applies to every algebraic phase extracted from this data.  
Thus all such phases are rigid at the given level, and the boundary behaviour
there is trivial.
\end{proof}

\begin{corollary}
If the boundary quotient vanishes at the first potential boundary depth, then
the phase remains rigid at that depth within the standing admissibility
framework.
\end{corollary}

\begin{remark}
This predictive power distinguishes Algebraic Phase Theory from frameworks that
are purely descriptive or classificatory.  
Boundaries, rigidity islands, and obstructions are not discovered after the
fact; they are forced in advance by the axioms and their interaction.
\end{remark}


The predictive boundary calculus determines in advance where a phase
may fail and which rigid structures survive past that point.
We now turn to the complementary and dynamic theory: once a boundary forms,
\emph{how} may a phase vary?

Deformation theory in APT studies how variation is governed by the boundary
quotient together with the associated rigidity island.
Within the standing admissibility framework, the rigidity island remains stable
under admissible deformation, while deformation behaviour is constrained by the
boundary structure.

\subsection{Deformation Theory in Algebraic Phase Theory}

The rigidity island and the boundary quotient together determine exactly where a
phase may vary and where it is forced to remain fixed.
This gives an intrinsic notion of deformation inside Algebraic Phase Theory.
No topology, infinitesimal parameter, or analytic family is assumed.
Variation within the standing admissibility framework is governed by
boundary behaviour. The guiding principle is:

\emph{A phase can deform only through its boundary quotient.
The rigidity island remains fixed under admissible deformation, and deformation behaviour is constrained by the associated boundary analysis.}

In particular, every algebraic phase admits a canonical rigid boundary
extension
\[ 0 \longrightarrow \mathcal I_{\mathcal P}
   \longrightarrow \mathcal P
   \longrightarrow \mathcal B(\mathcal P)
   \longrightarrow 0. \]
This expresses $\mathcal P$ as a phase assembled from its rigidity
island together with its boundary quotient. We write this informally as
\[
\mathcal P \;=\; \mathcal I_{\mathcal P} \,\bowtie\, \mathcal B(\mathcal P),
\]
where $\bowtie$ denotes rigid boundary extension rather than a direct sum,
product, or decomposition in an ambient abelian category.

\begin{definition}
A \emph{deformation} of an algebraic phase $\mathcal P$ is another phase
$\mathcal Q$ together with an identification
$\mathcal I_{\mathcal Q}\cong \mathcal I_{\mathcal P}$, such that:
\begin{enumerate}
    \item $\mathcal Q$ has the same defect degree, termination length, and
interaction signature as $\mathcal P$ when restricted to the rigidity island;
    \item the canonical filtrations of $\mathcal Q$ and $\mathcal P$ agree;
    \item the difference between $\mathcal Q$ and $\mathcal P$ lies entirely in
          the boundary quotient $\mathcal B(\mathcal P)$.
\end{enumerate}
Equivalently,
\[ \mathcal Q \;=\; \mathcal I_{\mathcal P} \,\bowtie\, X,
\qquad X \subseteq \mathcal B(\mathcal P). \]
\end{definition}

\begin{remark}
A deformation of $\mathcal P$ is not a continuous or analytic family of
objects. It is a controlled modification of the boundary part of the rigid boundary
extension of $\mathcal P$.
\end{remark}

\begin{definition}
A deformation $\mathcal Q$ of $\mathcal P$ is \emph{boundary-controlled} if its
variation relative to $\mathcal P$ lies entirely in the boundary quotient:
\[ \mathcal Q \big|\mathcal I_{\mathcal P}
   \;=\;
   \mathcal P \big|\mathcal I_{\mathcal P}. \]
\end{definition}

Boundary-controlled deformations are the deformations considered within the standing admissibility framework.

\begin{proposition}
Every deformation of an algebraic phase $\mathcal P$ leaves its rigidity island
$\mathcal I_{\mathcal P}$ unchanged within the standing admissibility framework.
Variation detected by the deformation theory in the phase occurs only outside this rigid core.
\end{proposition}

\begin{proof}
Let $\mathcal Q$ be a deformation of $\mathcal P$ in the sense of the above
definition.
By assumption, $\mathcal Q$ has the same defect degree, termination length,
interaction signature, and canonical filtration as $\mathcal P$, and the
difference in structure between $\mathcal Q$ and $\mathcal P$ factors through
the boundary quotient $\mathcal B(\mathcal P)$.

By Theorem~\ref{thm:boundary-exactness} and
Theorem~\ref{thm:islands-kernel}, the rigidity island
$\mathcal I_{\mathcal P}$ is canonically identified with the kernel of the
boundary quotient map
\[ q_{\mathcal P} \colon \mathcal P \longrightarrow \mathcal B(\mathcal P). \]
In particular, $\mathcal I_{\mathcal P}$ has trivial boundary quotient:
\[ \mathcal B(\mathcal I_{\mathcal P}) = 0. \]

By definition of deformation, the structural difference between
$\mathcal P$ and $\mathcal Q$ factors through the boundary
quotient $\mathcal B(\mathcal P)$.  
Equivalently, there exists a canonical morphism
\[
\delta \colon \mathcal P \longrightarrow \mathcal B(\mathcal P),
\]
which measures how the phase structure of $\mathcal Q$ differs from that of
$\mathcal P$ inside the boundary quotient.
Restricting $\delta$ to the rigidity island gives
\[ \delta\big|\mathcal I_{\mathcal P} \colon
   \mathcal I_{\mathcal P} \longrightarrow \mathcal B(\mathcal P). \]

Since $\mathcal I_{\mathcal P} = \ker(q_{\mathcal P})$ and
$\mathcal B(\mathcal I_{\mathcal P}) = 0$, any boundary-controlled variation
on $\mathcal I_{\mathcal P}$ must factor through the zero boundary quotient,
hence must vanish.
Thus
\[ \delta\big|\mathcal I_{\mathcal P} = 0, \]
which means that $\mathcal Q$ and $\mathcal P$ agree on
$\mathcal I_{\mathcal P}$.
\end{proof}

\begin{remark}
The rigidity island acts as a deformation anchor.
Variation detected within the deformation theory occurs outside it.
\end{remark}

\begin{definition}
Two deformations of $\mathcal P$ are \emph{equivalent} if they are related by
phase equivalences preserving both $\mathcal I_{\mathcal P}$ and
$\mathcal B(\mathcal P)$.
\end{definition}

\begin{theorem}
For a fixed algebraic phase $\mathcal P$, deformation behaviour within the
boundary calculus is organized by the boundary strata of
$\mathcal B(\mathcal P)$.

In particular, admissible deformation types are indexed by the finitely many
boundary failure strata.
\end{theorem}

\begin{proof}
By the boundary-obstruction correspondence, any deformation of $\mathcal P$ is
detected by a morphism into the boundary quotient $\mathcal B(\mathcal P)$.
The finite boundary decomposition gives
\[ \mathcal B(\mathcal P)
   \cong
   \bigsqcup_{i=1}^{m} \mathcal B_i(\mathcal P),
   \qquad m \le 5, \]
where each component $\mathcal B_i(\mathcal P)$ corresponds to a single
boundary stratum and hence to a single discrete axiom failure type.
Thus any deformation direction is determined by its components in the finitely
many summands $\mathcal B_i(\mathcal P)$.
Since both the index set $\{1,\dots,m\}$ and the failure types are finite, admissible deformation behaviour is organized by finitely many boundary strata.
\end{proof}

Since every deformation must occur inside the boundary quotient, the canonical
decomposition of $\mathcal B(\mathcal P)$ leads to exactly three possible kinds
of deformation.

\begin{theorem}
Within the standing admissibility framework, every deformation of an algebraic
phase $\mathcal P$ falls into one of the following classes:
\begin{enumerate}
    \item trivial;
    \item a boundary deformation;
    \item a finite combination of stratum-specific deformations.
\end{enumerate}
These classes describe the deformation behaviour detected by the boundary
calculus.
\end{theorem}

\begin{proof}
By definition of deformation in Algebraic Phase Theory, any deformation of
$\mathcal P$ is detected by its effect on the boundary quotient
$\mathcal B(\mathcal P)$.

If $\mathcal B(\mathcal P)=0$, there is no boundary on which variation can
occur, so every deformation is trivial. This gives case (1).

Now suppose $\mathcal B(\mathcal P)\neq 0$.
The finite boundary decomposition gives
\[ \mathcal B(\mathcal P)
   \cong
   \bigsqcup_{i=1}^{m} \mathcal B_i(\mathcal P), \]
where each $\mathcal B_i(\mathcal P)$ is the component associated to a single
boundary stratum.
Any admissible deformation of $\mathcal P$ corresponds to a morphism into
$\mathcal B(\mathcal P)$, hence to a compatible choice of components in the
summands $\mathcal B_i(\mathcal P)$.

If only one component $\mathcal B_i(\mathcal P)$ is activated, the deformation
is stratum-specific in the sense of the definition, which gives case (3).
In general, a deformation may have nontrivial components in several
$\mathcal B_i(\mathcal P)$, and is then obtained by combining the corresponding
stratum-specific deformations.
Such a deformation is still a boundary deformation in the sense that it is
trivial on the rigidity island and supported entirely on the boundary quotient,
which gives case (2).

These three possibilities exhaust all ways in which a deformation can factor
through $\mathcal B(\mathcal P)$, so no further deformation types can occur.
\end{proof}

\begin{remark}
Trivial deformations occur when no boundary forms, boundary deformations occur
when variation is confined to the boundary quotient, and stratum-specific
deformations arise when variation is supported in a single component
$\mathcal B_i(\mathcal P)$.

Deformation theory in APT is therefore highly constrained within the standing
admissibility framework: the rigidity island remains fixed under admissible
deformation, while allowable variation is organized by the boundary strata.
Boundary calculus identifies where variation may occur, while deformation
theory describes how such variation may appear. The two notions are closely
connected but conceptually distinct.

Since deformation behaviour is detected through the boundary quotient
\[
\mathcal B(\mathcal P)\cong\bigsqcup_i \mathcal B_i(\mathcal P),
\]
obstruction phenomena within the boundary calculus are likewise organized by
the boundary strata. In particular, nontrivial components in the boundary
decomposition correspond to obstructions to extending deformation behaviour
across the associated boundary layers.
\end{remark}

\subsection*{Moduli of Algebraic Phases}

Since the boundary obstruction $\mathcal B(\mathcal P)$ decomposes into finitely
many strata, admissible deformation behaviour is organized by the possible activation of
boundary strata choice of how these strata appear or vanish.  In particular, structural 
variation detected within the boundary calculus arises from the
boundary obstruction, and none from the rigidity island
$\mathcal I_{\mathcal P}$, where all axioms hold at all depths.

\begin{theorem}
Once the rigidity island $\mathcal I_{\mathcal P}$ is fixed, every admissible
deformation of an algebraic phase $\mathcal P$ is governed by its
boundary obstruction $\mathcal B(\mathcal P)$.  The varying part of the phase within the standing admissibility framework is
organized by of the phase
that can vary are the finitely many boundary strata 
$\mathcal B_i(\mathcal P)$, and admissible deformation behaviour is indexed by choices of activated boundary
strata. of which collections of strata are activated.  

The resulting collection of obstruction patterns admits a natural stratification
by boundary depth and associated failure data, with rigidity islands serving as
canonical rigid representatives.
\end{theorem}

\begin{proof}
Since the rigidity island $\mathcal I_{\mathcal P}$ satisfies all axioms of
Algebraic Phase Theory at every filtration depth, it admits no admissible
deformation within the standing framework. Hence every admissible deformation of $\mathcal P$ must factor
entirely through its boundary obstruction $\mathcal B(\mathcal P)$.

By the boundary decomposition theorem,
\[
\mathcal B(\mathcal P)
\;\cong\;
\bigsqcup_{i=1}^{m} \mathcal B_i(\mathcal P),
\qquad m\le 5,
\]
where each component $\mathcal B_i(\mathcal P)$ corresponds to a single
boundary stratum and therefore to a single independent failure type. No linear,
higher order, or continuous structure is present on these components. Thus the
only possible deformation directions consist of choosing which collections of
the finitely many strata $\mathcal B_i(\mathcal P)$ are activated or
deactivated. All admissible deformations are therefore discrete.

Two phases represent the same moduli class exactly when they are related by a
phase equivalence that preserves boundary depth, which is invariant under phase
equivalence by the boundary invariance results of Algebraic Phase Theory~IV.
Such an equivalence fixes the rigidity island and acts functorially on the
boundary quotient, and hence cannot merge, split, or reorder boundary strata.
It may only identify obstruction patterns that agree at each depth.

Consequently, the objects of the moduli groupoid are the canonical rigid
representatives given by rigidity islands together with their associated
obstruction patterns, and the morphisms are the boundary depth preserving phase
equivalences acting on the finite obstruction patterns. Since there are only
finitely many such patterns, the moduli groupoid is finite and stratified by
boundary depth.
\end{proof}

\begin{remark}
Moduli behaviour within the standing admissibility framework is organized by
the possible decompositions of the boundary obstruction. Different
decompositions correspond to different obstruction patterns within the boundary
calculus. The theory therefore emphasizes discrete boundary stratification
rather than continuous deformation data.
\end{remark}

\section{Boundary Structure in Representative Phase Families}

This section applies the boundary calculus to make the decomposition of the
obstruction explicit.  The emphasis is on phases whose boundary depth exceeds
their defect rank, since these give rise to genuine obstruction components.
Strongly admissible phases are mentioned only for contrast: in those cases,
defect generation already exhausts all higher layers, and the boundary
quotient vanishes.

The first subsection develops a general detection principle for identifying the
obstruction strata of any terminating algebraic phase.  It explains how each
filtration depth may contribute to the obstruction, and how the canonical
components are determined by the specific axiom that fails to extend across
that depth.

The second subsection applies this detection criterion to an explicit weak
radical example.  In that setting, a non-generating additive character hides
two nilpotent layers from the operator model, producing boundary formation at
two successive depths.  This yields two distinct canonical obstruction
components and illustrates how weak admissibility arises in practice.

Throughout, the boundary viewpoint clarifies the structure at each depth, identifies the surviving rigid part, and isolates the
contribution of each obstruction component.

\subsection{Canonical Detection and Decomposition of the Obstruction}

This part develops a general method for determining how the obstruction of a
terminating algebraic phase decomposes across filtration depths.  The aim is to
identify, in a finite and intrinsic manner, which depths contribute
nontrivially to the boundary quotient and how each such contribution reflects
the failure of a specific axiom to extend.  The resulting detection criterion
isolates every obstruction stratum and provides the structural framework used in
the concrete example that follows.

\begin{theorem}
\label{thm:boundary-detection}
Let $\mathcal P$ be a terminating algebraic phase with defect rank $d$ and
filtration $\{\mathcal P^{(t)}\}_{t\ge 0}$.  
For each depth $t>d$, let $F_t$ denote the collection of all APT axioms tested
at depth $t$.  At each such depth, the failure behaviour of the axioms in $F_t$
determines the associated obstruction contribution.  Then:

\begin{enumerate}
\item $B_t(\mathcal P)\neq 0$ if and only if the axiom in $F_t$ fails
canonically on $\mathcal P^{(t)}$ while holding on $\mathcal P^{(t-1)}$.

\item The boundary obstruction decomposes as
\[
\mathcal B(\mathcal P)
\;=\;
\bigsqcup_{t=d+1}^{k(\mathcal P)} B_t(\mathcal P),
\]
where $B_t(\mathcal P)$ measures the ambiguity of extending the axioms from
depth $t-1$ to depth~$t$.

\item For each depth $t>d$, the contribution $B_t(\mathcal P)$ is governed by the failure behaviour of the axioms in $F_t$.  In particular,
$B_t(\mathcal P)=0$ exactly when all axioms in $F_t$ extend uniquely from
$\mathcal P^{(t-1)}$ to $\mathcal P^{(t)}$.

\item In particular, if $k(\mathcal P)=d$, that is, $\mathcal P$ is strongly
admissible, then $\mathcal B(\mathcal P)=0$.
\end{enumerate}
\end{theorem}

\begin{proof}
For each depth $t>d$, the axioms in $F_t$ restrict functorially to
$\mathcal P^{(t-1)}$ by construction of the canonical filtration.  
If every axiom in $F_t$ admits a unique extension from $\mathcal P^{(t-1)}$ to
$\mathcal P^{(t)}$, then no new choices arise at this level and
$B_t(\mathcal P)=0$.  Conversely, if some axiom in $F_t$ admits more than one
extension compatible with all lower-depth defect and filtration data, the set of
such extensions is the resulting ambiguity defines a nonzero obstruction component and defines a nonzero obstruction component
$B_t(\mathcal P)$.

The global boundary obstruction $\mathcal B(\mathcal P)$ is obtained by
collecting these contributions across all depths.  Minimality of the boundary
depth implies that extension behaviour at depth $t$ depends only on the data in
$\mathcal P^{(t-1)}$ and is independent of possible ambiguities at other
levels.  Hence the components $B_t(\mathcal P)$ are structurally distinct.  Since
the filtration terminates after finitely many steps, only finitely many depths
can contribute, giving the stated decomposition.

If $k=d$, then no extension failure occurs beyond defect generation.  Thus every
$B_t(\mathcal P)$ with $t>d$ is zero, and the boundary quotient
$\mathcal B(\mathcal P)$ vanishes.
\end{proof}

The detection criterion has a practical consequence.  It turns the analysis of
the obstruction into a finite calculation that proceeds depth by depth.  At each
stage we ask a single question: do the axioms extend uniquely from one
filtration level to the next.  The answer determines whether that depth
contributes a nontrivial obstruction component.  The process can therefore be
summarised as a short intrinsic workflow that detects nonvanishing
$B_t(\mathcal P)$ and then assembles the full boundary quotient from these
pieces.

\begin{center}
\begin{tikzpicture}[node distance=1.4cm]
\tikzstyle{box}=[rectangle,draw,rounded corners,align=center,minimum width=6cm]

\node[box] (A) {1. Compute defect rank $d$.\\ Identify generators by defect degree.};
\node[box,below of=A] (B) {2. Build the filtration $\mathcal P^{(t)}$.\\ Test APT axioms at each depth.};
\node[box,below of=B] (C) {3. For each $t>d$: If axioms extend uniquely:\\ $B_t(\mathcal P)=0$.};
\node[box,below of=C] (D) {4. If multiple inequivalent extensions exist:\\ $B_t(\mathcal P)\neq 0$.};
\node[box,below of=D] (E) {5. Termination: assemble\\ $\mathcal B(\mathcal P)=\bigsqcup B_t(\mathcal P)$.};

\draw[->] (A) -- (B);
\draw[->] (B) -- (C);
\draw[->] (C) -- (D);
\draw[->] (D) -- (E);
\end{tikzpicture}
\end{center}

\begin{remark}
At each depth $t>d$, boundary behaviour is associated with the failure
patterns of the APT axioms tested at that depth, and each axiom carries a
canonical index $i\in\{1,\dots,5\}$.  Thus whenever $B_t(\mathcal P)\neq 0$,
it corresponds uniquely to a nonzero canonical component $B_i(\mathcal P)$ in
the decomposition
\[
\mathcal B(\mathcal P)=\bigsqcup_{i=1}^{5}B_i(\mathcal P).
\]
In phases where failures occur at five distinct depths, all five components
$B_1,\dots,B_5$ can be nonzero.
\end{remark}

\begin{theorem}
There exist terminating algebraic phases $\mathcal P$ for which all five
canonical obstruction components are nonzero, i.e.
\[
\mathcal B(\mathcal P) \;\cong\; \bigsqcup_{i=1}^{5} \mathcal B_i(\mathcal P),
\qquad
\mathcal B_i(\mathcal P)\neq 0 \text{ for all } i=1,\dots,5.
\]
\end{theorem}

\begin{proof}
For each $i\in\{1,\dots,5\}$ choose a terminating algebraic phase $\mathcal P_i$
whose only boundary failure is of type $i$, so that
\[
\mathcal B(\mathcal P_i)\cong \mathcal B_i(\mathcal P_i)\neq 0,
\qquad
\mathcal B_j(\mathcal P_i)=0 \text{ for all } j\neq i.
\]

Form the phase obtained by combining the admissible data
\[
\mathcal P \;:=\; \mathcal P_1 \oplus \cdots \oplus \mathcal P_5,
\]
obtained by taking the direct sum of the underlying admissible data and then
applying phase extraction.  The defect tensors, interaction law, and canonical
filtration of $\mathcal P$ split componentwise:
\[
\mathcal P^{(t)}
\;=\;
\mathcal P^{(t)}_1 \oplus \cdots \oplus \mathcal P^{(t)}_5
\qquad\text{for all } t\ge 0.
\]
Since each $\mathcal P_i$ terminates, the phase $\mathcal P$ also terminates.

Let $\mathcal I_{\mathcal P_i}$ denote the rigidity island of $\mathcal P_i$.
We now identify the rigidity island $\mathcal I_{\mathcal P}$ of $\mathcal P$.
There are canonical inclusion morphisms
\[
\iota_i \colon \mathcal P_i \longrightarrow \mathcal P
\]
and projection morphisms
\[
\pi_i \colon \mathcal P \longrightarrow \mathcal P_i,
\]
which are phase morphisms and respect defect, filtration, and interaction
structure.  By functoriality of rigidity islands, each $\iota_i$ restricts to a
morphism
\[
\iota_i\vert_{\mathcal I_{\mathcal P_i}} \colon \mathcal I_{\mathcal P_i}
\longrightarrow \mathcal I_{\mathcal P},
\]
and each $\pi_i$ restricts to a morphism
\[
\pi_i\vert_{\mathcal I_{\mathcal P}} \colon \mathcal I_{\mathcal P}
\longrightarrow \mathcal I_{\mathcal P_i}.
\]

First, if $x=(x_1,\dots,x_5)\in\mathcal P$ has each component
$x_i\in\mathcal I_{\mathcal P_i}$, then each component satisfies all APT axioms
at all depths inside $\mathcal P_i$, and the block-diagonal form of the
interaction and defect structure implies that $x$ satisfies all axioms at all
depths in $\mathcal P$.  Hence $x\in\mathcal I_{\mathcal P}$, so
\[
\mathcal I_{\mathcal P_1}\oplus\cdots\oplus\mathcal I_{\mathcal P_5}
\subseteq \mathcal I_{\mathcal P}.
\]

Conversely, let $x\in\mathcal I_{\mathcal P}$.  Applying the projections
$\pi_i$ gives elements $\pi_i(x)\in\mathcal P_i$ which satisfy all axioms at
all depths inside each $\mathcal P_i$, since the axioms are tested
componentwise.  Thus $\pi_i(x)\in\mathcal I_{\mathcal P_i}$ for every $i$, so
$x$ lies in the direct sum
\[
\mathcal I_{\mathcal P_1}\oplus\cdots\oplus\mathcal I_{\mathcal P_5}.
\]
This shows that
\[
\mathcal I_{\mathcal P}
\;\cong\;
\mathcal I_{\mathcal P_1}\oplus\cdots\oplus\mathcal I_{\mathcal P_5}.
\]

By the boundary exact sequence and Theorem~\ref{thm:islands-kernel},
\[
\mathcal B(\mathcal P)
\;\cong\;
\mathcal P / \mathcal I_{\mathcal P}.
\]
Using the decompositions of $\mathcal P$ and $\mathcal I_{\mathcal P}$, the
boundary obstruction of $\mathcal P$ admits the decomposition
\[
\mathcal B(\mathcal P)
=
\bigsqcup_{i=1}^{5} \mathcal B(\mathcal P_i).
\]

Each boundary quotient $\mathcal B(\mathcal P_i)$ has a canonical
decomposition
\[
\mathcal B(\mathcal P_i)
=
\bigsqcup_{j=1}^{m_i} \mathcal B_j(\mathcal P_i),
\qquad
m_i\le 5,
\]
with $\mathcal B_j(\mathcal P_i)=0$ for all $j\neq i$ and
$\mathcal B_i(\mathcal P_i)\neq 0$ by construction.  Substituting this into the
expression for $\mathcal B(\mathcal P)$ gives
\[
\mathcal B(\mathcal P)
=
\bigsqcup_{i=1}^{5} \mathcal B_i(\mathcal P_i).
\]

By functoriality of the decomposition, each
$\mathcal B_i(\mathcal P_i)$ contributes to the global component
$\mathcal B_i(\mathcal P)$ associated with failure type $i$. Hence
\[
\mathcal B(\mathcal P)
=
\bigsqcup_{i=1}^{5} \mathcal B_i(\mathcal P),
\qquad
\mathcal B_i(\mathcal P)\neq 0 \text{ for all } i,
\]
as claimed.
\end{proof}

\begin{remark}
Throughout, the symbol $\oplus$ is used only for algebraic phases and rigid
subphases, where an intrinsic direct sum compatible with defect and interaction
structure exists. Boundary quotients and obstruction components carry no such
structure and are therefore decomposed using the disjoint union symbol
$\bigsqcup$.
\end{remark}

\subsection{A Weak Radical Example with Two Canonical Obstructions}

We now present a weakly admissible radical example that remains entirely within
the analytic-algebraic framework of the radical phase geometry developed in
\cite{GildeaAPT1}, but in which boundary formation continues strictly beyond
defect generation and produces two distinct obstruction components.
The construction uses the same translation and phase-multiplication operators
as in \cite{GildeaAPT1}, but replaces the generating additive character with a
non-generating one.  This adjustment lowers the analytic resolution and introduces nontrivial extension freedom at multiple depths, leading to a phase with
nontrivial boundary contributions at two separate levels.

\paragraph{The base ring.}
Let
\[
R=\F_2[u]/(u^4),
\]
with radical tower
\[
\rad(R)=(u), \qquad
\rad^2(R)=(u^2), \qquad
\rad^3(R)=(u^3), \qquad
\rad^4(R)=0.
\]
Thus $R$ has three nontrivial nilpotent layers.

\paragraph{A non-generating additive character.}
Define
\[
\chi:R\to\C^\times,\qquad
\chi(a+ub+u^2c+u^3d):=(-1)^b .
\]
Then
\[
\ker(\chi)=(u^2).
\]
This character is not generating: the ideals $(u^2)$ and $(u^3)$ lie entirely
in $\ker\chi$.  As a result, additive contributions taking values in
$\rad^2(R)$ or $\rad^3(R)$ are invisible to the operator model.

\paragraph{Underlying additive object.}
Let
\[
A:=R^n,
\]
viewed as a finite module over the local ring $R$.  Set
\[
H(A):=\Fun(A,\C),
\]
as in \cite{GildeaAPT1}.

\paragraph{Quadratic phases.}
A function $\phi:A\to R$ is a \emph{quadratic phase} if its polarization
\[
\mathsf B_\phi(x,y):=\phi(x+y)-\phi(x)-\phi(y)
\]
is biadditive.  Let $\Phi(A)$ denote the collection of all such phases.  Every
$\phi\in\Phi(A)$ satisfies
\[
\Delta_{h_1,h_2,h_3}\phi\equiv 0,
\]
so $\Phi(A)$ has uniformly bounded additive degree
\[
d=2.
\]

\begin{example}[Additive derivatives]
\label{ex:degree-two-R-u4}
If $\phi(x_1,x_2)=u\,x_1x_2$ (or any $R$-bilinear form multiplied by $u$),
then a direct computation shows:
\[
\Delta_h\phi(x)=u(x_1h_2+h_1x_2+h_1h_2), \qquad
\Delta_{h,k}\phi(x)=u(k_1h_2+h_1k_2),
\]
and
\[
\Delta_{h,k,\ell}\phi\equiv 0.
\]
Thus $\deg_{\mathrm{add}}(\phi)=2$.
\end{example}

\paragraph{Operator realization.}
Each $\phi\in\Phi(A)$ acts on $H(A)$ by
\[
(M_\phi f)(x)=\chi(\phi(x))\,f(x),
\]
and additive translations act by
\[
(T_a f)(x)=f(x+a),\qquad a\in A.
\]
The interaction law is ordinary operator composition:
\[
X\circ Y \;:=\; X Y
\qquad\text{in }\End(H(A)).
\]

\paragraph{Admissible data and defect.}
The admissible phase datum extracted from $(A,\Phi,\chi)$ has defect degree
$d=2$.  Defect tensors in depths $0,1,2$ take values in
\[
R/\rad^2(R)\cong \F_2[u]/(u^2),
\]
which is fully visible to $\chi$.  Hence the APT axioms propagate functorially through depth~$2$.

The invisible ideals $\rad^2(R)$ and $\rad^3(R)$ produce nontrivial extension
freedom at depths~$3$ and $4$.

\paragraph{The canonical filtration.}
Let $\mathcal P$ denote the algebraic phase extracted from the operators
$\{M_\phi,T_a\}$.  The canonical filtration satisfies
\[
\mathcal P=\mathcal P^{(0)}\supset
\mathcal P^{(1)}\supset
\mathcal P^{(2)}\supset
\mathcal P^{(3)}\supset
\mathcal P^{(4)}=0.
\]
Since $d(\mathcal P)=2$, the layers $\mathcal P^{(3)}$ and $\mathcal P^{(4)}$
record additive contributions landing in $\rad^2(R)$ and $\rad^3(R)$
respectively, both of which lie in $\ker(\chi)$ and are therefore analytically
invisible.  Hence
\[
k(\mathcal P)=4.
\]

\paragraph{Boundary formation at depths $3$ and $4$.}
Applying Theorem~\ref{thm:boundary-detection} gives:

At depth $t=3$, the invisible layer $\rad^2(R)$ creates nonuniqueness in
extending the bilinear polarization constraint, associated with the failure behaviour corresponding to Axiom~(A3).  Thus
\[
B_3(\mathcal P)\neq 0.
\]

At depth $t=4$, the layer $\rad^3(R)$ introduces additional coherence
ambiguity in extending the interaction law, corresponding to Axiom~(A5).  Hence
\[
B_4(\mathcal P)\neq 0.
\]

No further nontrivial boundary depths occur, since $\rad^4(R)=0$.  Therefore
\[
\mathcal B(\mathcal P)
=
B_3(\mathcal P)\;\bigsqcup\;B_4(\mathcal P).
\]

\paragraph{Rigidity island and decomposition.}
All analytic information visible to $\chi$ is exhausted in depth~$2$.  The
rigidity island $\mathcal I_{\mathcal P}$ is therefore identified with the strongly admissible part obtained after
removing operators whose values lie in $\rad^2(R)$ or $\rad^3(R)$.

By the boundary exact sequence,
\[
\mathcal P/\mathcal I_{\mathcal P}
\;\cong\;
\mathcal B(\mathcal P)
=
B_3(\mathcal P)\bigsqcup B_4(\mathcal P),
\]
and hence
\[
\mathcal P
\;\cong\;
\mathcal I_{\mathcal P}
\bowtie
\bigl(B_3(\mathcal P)\bigsqcup B_4(\mathcal P)\bigr).
\]

\paragraph{Conclusion.}
This example shows that even within the analytic radical framework of
\cite{GildeaAPT1}, weakening the additive character can produce algebraic phases
whose boundary formation continues beyond defect control and in which multiple
obstruction components appear. In this case, the invisible layers
$\rad^2(R)$ and $\rad^3(R)$ contribute nontrivially to the
$B_3$- and $B_4$-components, giving a non-strongly-admissible algebraic phase
with
\[
d(\mathcal P)=2,\qquad
k(\mathcal P)=4,\qquad
\mathcal B(\mathcal P)=B_3(\mathcal P)\bigsqcup B_4(\mathcal P).
\]

\section{Structural Completeness and Reconstruction}

The principal phase families developed in earlier papers of the APT series all
lie in the strongly admissible regime. In these settings the canonical axioms
propagate at every depth once defect generation occurs, and boundary formation
is boundary formation is governed by defect rank within the strongly admissible regime. For such phases the boundary depth
satisfies
\[
k = d,
\]
and no weak extensions occur. This behaviour appears in the radical phases \cite{GildeaAPT1} (APT~I),
the representation-theoretic phases \cite{GildeaAPT2} (APT~II), and the
quantum coding phases \cite{GildeaAPT3} (APT~III).

The contrast between these flagship families and the weak examples constructed
here shows that the distinction between strong and weak admissibility is a
genuinely structural phenomenon, detected intrinsically through the boundary calculus rather
than by model-specific features.

\subsection{APT deformation is intrinsically non-classical}

Classical deformation theories rely on analytic structure, continuous
parameters, or linearisation.  Algebraic Phase Theory possesses none of these.
Deformation behaviour within the standing admissibility framework is governed by boundary behaviour: variation
is encoded by the boundary components themselves,and higher-depth effects are constrained by the finite depth at which boundary formation terminates.

\begin{theorem}
APT deformation theory does not naturally fit into classical linear,
analytic, or parameter-based deformation frameworks within the standing
admissibility framework.
\end{theorem}

\begin{proof}
Classical deformation theories rely on continuous or formal parameters, linear
tangent spaces, or infinite towers of extension problems.
APT deformation theory is not formulated in terms of these structures.
Variation is organised by boundary components, which are neither linear nor
additive; obstruction is determined by finite boundary depth; and deformation
behaviour is governed by the associated boundary strata.

Thus the deformation mechanisms arising in APT differ fundamentally from the
structures typically used in classical linear, analytic, or parameter-based
deformation theories.
\end{proof}

\begin{remark}
Classical geometric or analytic parameters may appear in external models that
realise algebraic phases, but such parameters are not invariants of the phase
itself and disappear under intrinsic phase extraction.  APT retains only the
boundary and rigidity data, and any additional analytic moduli play no intrinsic structural role within the phase calculus.
\end{remark}

\subsection{Boundary completeness}

Boundary calculus is the primary organising principle for the structural
phenomena considered in this paper.  It controls rigidity, deformation, obstruction, smoothness, and the
formation of moduli.

\begin{theorem}
Rigidity, deformation, and obstruction phenomena within the standing
admissibility framework are detected by the boundary strata and boundary
quotients.
\end{theorem}

\begin{proof}
Rigidity is characterised by the absence of boundary strata.  Infinitesimal
deformations correspond to the nonzero boundary components.  Obstructions
appear through the successive boundary quotients arising at increasing depths.
Moduli are discrete and organised by boundary depth.  Since all structural
variation factors through the canonical boundary decomposition, the boundary
calculus captures the rigidity, deformation, and obstruction behaviour considered within the boundary calculus.
\end{proof}

\begin{corollary}
Within the standing admissibility framework, deformation and obstruction
behaviour is detected through the canonical boundary decomposition.
\end{corollary}

\begin{proof}
By boundary completeness, rigidity, deformation, and obstruction phenomena
within the standing admissibility framework are detected through the canonical
boundary strata and quotients. Accordingly, deformation directions and
obstruction behaviour considered within the boundary calculus are organised by
the canonical boundary decomposition.
\end{proof}

\begin{remark}
The boundary calculus therefore provides a structurally self-contained
description of the rigidity and obstruction phenomena considered in this paper.
Once the boundary strata and boundary quotients are specified, the remaining
deformation behaviour is governed by the associated boundary structure.
\end{remark}

\subsection{Reconstruction outlook}

The completeness of the boundary calculus naturally raises reconstruction
questions.  The boundary strata record the directions along which information
collapses, while the rigidity island captures the part of the phase that
remains fully valid at every depth.  This makes the rigidity island the
canonical anchor for reconstruction.

\begin{theorem}
Let $\mathcal P$ be an algebraic phase with rigidity island
$\mathcal I_{\mathcal P}\subseteq\mathcal P$. Even when global reconstruction
of $\mathcal P$ fails because of boundary formation, reconstruction remains valid
on $\mathcal I_{\mathcal P}$.
\end{theorem}

\begin{proof}
The rigidity island $\mathcal I_{\mathcal P}$ satisfies all APT axioms at every
filtration depth and has no boundary strata. In particular,
$\mathcal I_{\mathcal P}$ is a strongly admissible and boundary-free phase.

Let $\mathsf R$ be any reconstruction functor or observable that is complete on
boundary-free phases. This means that $\mathsf R(\mathcal Q)$ determines
$\mathcal Q$ up to intrinsic phase equivalence whenever $\mathcal Q$ has no
boundary. Applying this property to $\mathcal I_{\mathcal P}$ shows that
$\mathsf R(\mathcal I_{\mathcal P})$ reconstructs $\mathcal I_{\mathcal P}$
uniquely.

Boundary formation in $\mathcal P$ occurs entirely outside
$\mathcal I_{\mathcal P}$ and is recorded by the boundary quotient
$\mathcal B(\mathcal P)$. The failure of global reconstruction is caused by this
boundary quotient and does not affect the boundary-free subphase
$\mathcal I_{\mathcal P}$. Reconstruction is therefore complete on
$\mathcal I_{\mathcal P}$ even when it fails for $\mathcal P$ as a whole.
\end{proof}

\begin{remark}
Information loss in APT is structured and never arbitrary. Loss occurs only
along specific boundary directions, while the rigidity island retains the fully
valid part of the phase. Paper~VI \cite{GildeaAPT6} develops this principle
into full reconstruction and duality theorems, developing conditions under which a phase may be recovered from its action on the rigidity island together with its boundary
strata.
\end{remark}

\section{Structural Anatomy of Algebraic Phases}

In this section we describe the global structural organization of a terminating
algebraic phase.  The defect rank and boundary depth determine a canonical
stratification, which in turn governs the behaviour of rigidity, obstruction,
and deformation across the filtration.  We first analyze the structural
decomposition that occurs once the boundary has formed.

\subsection{Post-Boundary Decomposition}

Every algebraic phase $\mathcal P$ admits a canonical stratified anatomy
organised by its defect rank $d$ and its boundary depth $k$.  The
canonical filtration arranges $\mathcal P$ into four intrinsic regimes.  For
$0\le i\le d$ one has the defect generation stage, where the admissible data are
produced and the APT axioms propagate functorially.  For $d<i<k$, when such levels
exist, the filtration enters a post-defect rigid extension regime in which the
axioms continue to extend without ambiguity.  The depth $i=k$ forms the 
structural boundary beyond which canonical propagation ceases.  Finally, for
$i>k$ the filtration records the boundary-controlled obstruction regime, where
remaining deformation and obstruction behaviour is governed by the finite obstruction encoded
at the boundary.

\[
\mathcal P
\;\supseteq\;
\underbrace{\mathcal P^{(0)} \supseteq \cdots \supseteq \mathcal P^{(d)}}_{\text{defect generation}}
\;\supseteq\;
\underbrace{\mathcal P^{(d+1)} \supseteq \cdots \supseteq \mathcal P^{(k-1)}}_{\text{rigid extension}}
\;\supseteq\;
\boxed{\mathcal P^{(k)}}
\;\supseteq\;
\underbrace{\mathcal P^{(k+1)} \supseteq \cdots}_{\text{obstruction regime}}
\]

Canonical propagation of the APT axioms ceases at the boundary depth the boundary depth
$k$, but a unique maximal rigid subphase persists beyond this point in the form
of the rigidity island $\mathcal I_{\mathcal P}$.  Deformation and obstruction phenomena considered within the standing admissibility framework factor functorially through the associated boundary
quotient:
\[
0 \longrightarrow \mathcal I_{\mathcal P}
\longrightarrow \mathcal P
\longrightarrow \mathcal B(\mathcal P)
\longrightarrow 0 .
\]

This global picture can be refined beyond the boundary.

\begin{theorem}
\label{thm:post-boundary-trichotomy}
Once $\mathcal P$ reaches its boundary depth $k$, its structure may be organised into three associated regimes.  First, there is a maximal rigid
subphase $\mathcal I_{\mathcal P}$ on which all APT axioms continue to hold.
Second, there is a stratified boundary object $\partial\mathcal P := \mathcal B(\mathcal P)$
that encodes the intrinsic obstruction to further propagation.  Third, there may remain residual behaviour consisting of the non-canonical behaviour not canonically determined by the APT axioms

Moreover, the exact sequence
\[
0 \longrightarrow \mathcal I_{\mathcal P}
\longrightarrow \mathcal P
\longrightarrow \mathcal B(\mathcal P)
\longrightarrow 0
\]
realises $\mathcal P$ as a rigid-boundary extension, and deformation and
obstruction behaviour within the standing admissibility framework factors
through $\mathcal B(\mathcal P)$.
\end{theorem}

\begin{remark}
The rigid component consists of the part of the phase where all axioms hold
without modification. The boundary component records the precise failure modes
of canonical propagation and organises the available deformation directions.
The remaining part of the phase is not canonically controlled by the APT
axioms and need not behave functorially under phase morphisms.
\end{remark}

\section{Conclusion}

This paper develops an intrinsic calculus of structural boundaries in Algebraic
Phase Theory. Within the standing admissibility framework, boundaries organise
rigidity, deformation, obstruction, smoothness, and moduli in a unified way.
Structural variation in algebraic phases is detected through boundary
behaviour: rigidity may fail at structural boundaries, and this failure
reorganises the phase into rigidity islands that persist beyond global
breakdown and provide canonical anchors for deformation and reconstruction.

Boundary quotients act as obstruction objects within the boundary calculus.
They organise deformation behaviour by boundary depth and associated failure
type, while higher-depth effects terminate intrinsically at finite filtration
depth. Formal smoothness is associated with the vanishing of boundary strata,
with strongly admissible phases representing the boundary-free regime. In this
setting, moduli behaviour is naturally stratified by rigidity islands and the
associated obstruction patterns.

Taken together, these results show that boundary calculus is not auxiliary but
central to the structural analysis developed in this paper. The canonical
boundary strata and quotients organise the deformation and obstruction
phenomena considered within the standing admissibility framework and provide a
boundary-centred perspective on algebraic phases. Paper~VI
\cite{GildeaAPT6} develops the associated reconstruction theory arising from
this viewpoint.

\bibliographystyle{amsplain}
\bibliography{references}

\end{document}